\documentclass[11pt]{amsart}   	
\usepackage{geometry}      
\usepackage[dvipsnames]{xcolor}   		             		
\usepackage{graphicx}				
\usepackage{amssymb,mathrsfs}
\usepackage{amsthm,stmaryrd}
\usepackage[leqno]{amsmath}
\usepackage{tikz}
\usepackage{tikz-cd}
\usepackage{accents,upgreek,enumerate}
\usepackage[headings]{fullpage}
\usepackage{bm}
\usepackage{mathtools}
\usepackage[all]{xy}
\usepackage{caption}
\usepackage{comment}
\usepackage{tabu}

\usepackage{colonequals}

\tikzset{
  commutative diagrams/.cd, 
  arrow style=tikz, 
  diagrams={>=stealth}
}

\newtheorem{innercustomthm}{Theorem}
\newenvironment{customthm}[1]
  {\renewcommand\theinnercustomthm{#1}\innercustomthm}
  {\endinnercustomthm}
  
  \makeatletter
\def\@tocline#1#2#3#4#5#6#7{\relax
  \ifnum #1>\c@tocdepth 
  \else
    \par \addpenalty\@secpenalty\addvspace{#2}%
    \begingroup \hyphenpenalty\@M
    \@ifempty{#4}{%
      \@tempdima\csname r@tocindent\number#1\endcsname\relax
    }{%
      \@tempdima#4\relax
    }%
    \parindent\z@ \leftskip#3\relax \advance\leftskip\@tempdima\relax
    \rightskip\@pnumwidth plus4em \parfillskip-\@pnumwidth
    #5\leavevmode\hskip-\@tempdima
      \ifcase #1
       \or\or \hskip 1em \or \hskip 2em \else \hskip 3em \fi%
      #6\nobreak\relax
    \dotfill\hbox to\@pnumwidth{\@tocpagenum{#7}}\par
    \nobreak
    \endgroup
  \fi}
\makeatother

\usetikzlibrary{calc}
\usetikzlibrary{fadings}
\usetikzlibrary{decorations.pathmorphing}
\usetikzlibrary{decorations.pathreplacing}

\newcounter{marginnote}
\DeclareMathAlphabet{\mathpzc}{OT1}{pzc}{m}{it}

\usepackage[backref=page]{hyperref}
\hypersetup{
  colorlinks   = true,          
  urlcolor     = violet,          
  linkcolor    = blue,          
  citecolor   = violet             
}

\theoremstyle{definition}

\newtheorem{theorem}{Theorem}[section]
\newtheorem{corollary}[theorem]{Corollary}
\newtheorem{lemma}[theorem]{Lemma}
\newtheorem{proposition}[theorem]{Proposition}

\newtheorem{quasi-theorem}[theorem]{Quasi-Theorem}

\theoremstyle{definition}
\newtheorem{definition}[theorem]{Definition}

\newtheorem{problem}[theorem]{Problem}

\newtheorem{example}[theorem]{Example}
\newtheorem{blank remark}[theorem]{}
\newtheorem{remark}[theorem]{Remark}

\newtheorem{not1}[theorem]{Notation}

\newcommand{\pt}{\operatorname{pt}}

\newcommand{\PP}{\mathbb{P}}         
\newcommand{\QQ} {{\mathbb Q}}		
\newcommand{\RR} {{\mathbb R}}

\def\R{\mathrm{R}}

\newcommand{\Star}{\operatorname{Star}}

\newcommand{\Hom}{\operatorname{Hom}}

\newcommand{\ev}{\operatorname{ev}}

\newcommand{\cal}{\mathcal}

\def\cM{{\cal M}}

\newcommand{\sqC}{\scalebox{0.75}[1.2]{$\sqsubset$}}

\newcommand{\Kup}{\mathsf{K}}
\newcommand{\Tup}{\mathsf{T}}

\newcommand{\Mbar}{\overline{\cM}\vphantom{\cM}}

\newcommand{\OO}{\mathcal{O}}

\newcommand{\virt}{\operatorname{vir}}

\makeatletter
\def\blfootnote{\xdef\@thefnmark{}\@footnotetext}
\makeatother

\title{Gromov--Witten theory with maximal contacts}

\date{}

\author{Navid Nabijou  {\it \&} Dhruv Ranganathan}

\address{Department of Pure Mathematics {\it \&} Mathematical Statistics\\
University of Cambridge}
\email{N.N. \href{mailto:nn333@cam.ac.uk}{nn333@cam.ac.uk}}
\email{D.R. \href{mailto:dr508@cam.ac.uk}{dr508@cam.ac.uk}}

\vspace{-0.2in}

\begin{document}

\begin{abstract}
We propose an intersection-theoretic method to reduce questions in genus zero logarithmic Gromov--Witten theory to questions in the Gromov--Witten theory of smooth pairs, in the presence of positivity. The method is applied to the enumerative geometry of rational curves with maximal contact orders along a simple normal crossings divisor and to recent questions about its relationship to local curve counting. Three results are established. We produce counterexamples to the local-logarithmic conjectures of van Garrel--Graber--Ruddat and Tseng--You. We prove that a weak form of the conjecture holds for product geometries. Finally, we explicitly determine the difference between local and logarithmic theories, in terms of relative invariants for which efficient algorithms are known. The polyhedral geometry of the tropical moduli of maps plays an essential and intricate role in the analysis. 
\end{abstract}

\maketitle
\setcounter{tocdepth}{1}
\tableofcontents

\section*{Introduction}
\noindent Let $X$ be a smooth projective variety and $D$ a simple normal crossings divisor whose irreducible components $D_1,\ldots,D_k$ are hyperplane sections, hereafter \textit{section pairs}.  We examine three genus zero Gromov--Witten theories: (1) the logarithmic theory of $(X|D)$, (2) the {naive} logarithmic theory of $(X|D)$ constructed out of the relative theories of $(X|D_i)$, and (3) the local theory of the direct sum of the $\mathcal O_X(-D_i)$. The first two encode rational curves in $X$ with tangency conditions along $D$. The local theory models rational curves in a rigid embedding of $X$ in an ambient variety with split normal bundle $\oplus_{i=1}^k \mathcal O_X(-D_i)$. 

The \textit{naive theory} is defined as follows. First let $X=\PP$ be a product of $k$ projective spaces, with $H_i \subseteq \PP$ the pullback of a hyperplane from the $i$th factor and $H=\Sigma_{i=1}^k H_i$. The \emph{naive space} is defined as the fibre product of stacks:
\begin{equation*}
\begin{tikzcd}
\mathsf{N}(\PP|H) \ar[r] \ar[d] \ar[rd,phantom,"\square"] & \prod_{i=1}^k \Kup(\PP|H_i) \ar[d] \\
\Kup(\PP) \ar[r,hook,"\Delta"] & \prod_{i=1}^k \Kup(\PP).
\end{tikzcd}
\end{equation*}
The moduli space $\Kup(\PP)$ is smooth and so $\Delta$ is a regular embedding. We obtain a virtual class on $\mathsf{N}(\PP|H)$ by pullback:
\begin{equation*} [\mathsf{N}(\PP|H)]^{\virt} \colonequals \Delta^! \left( \prod_{i=1}^k [\Kup(\PP|H_i)] \right). \end{equation*}
 The pushforward to $\Kup(\PP)$ is simply the product of the classes $[\Kup(\PP|H_i)]$. Virtual pullback defines the theory for arbitrary section pairs $(X|D)$, see \S\ref{sec: virtual pullback}.\footnote{A fourth possibility is the multi-root theory of $(X|D)$ constructed by Tseng--You~\cite{TY20}. This coincides with the naive theory for section pairs, see \cite[Corollary 2.2]{BNTY}.}

\subsection{Correspondences} If $D$ is smooth, logarithmic Gromov--Witten theory coincides with the relative theory for all tangency orders. If the tangency with $D$ is maximal, it coincides with the local theory by a result of van Garrel--Graber--Ruddat~\cite{vGGR}, following Takahashi and Gathmann~\cite{Tak01,GathmannRelative}.

\noindent
{\bf The local-logarithmic correspondences.} {\it Let $X$ be smooth and projective with simple normal crossings divisor $D$ with smooth nef components $D_1,\ldots, D_k$. Let $\beta$ be a curve class and suppose $d_i \colonequals D_i \cdot \beta >0$ for all $i$. Consider the moduli of logarithmic stable maps $\Kup^{\max}_{0,k}(X|D,\beta)$ with maximal contact with each component at distinct points. 

\noindent
\underline{\smash{Strong form:}} There is an equality of homology classes on the Kontsevich space $\Kup_{0,k}(X,\beta)$ of $k$-pointed stable maps to $X$, suppressing the relevant pushforwards, given by:
\[
[\Kup^{\mathrm{max}}_{0,k}(X|D,\beta)]^{\mathrm{vir}} =  \prod_{i=1}^k (-1)^{d_i+1} \mathsf{ev}_i^\star D_i \cdot [\Kup_{0,k}(\oplus_{i=1}^k \OO_X(-D_i),\beta)]^{\mathrm{vir}}.
\]

\noindent
\underline{\smash{Original form:}} The equality above holds after pushing forward to the Kontsevich space $\Kup_{0,0}(X,\beta)$ of unpointed stable maps to $X$, that is:
\[
[\Kup^{\mathrm{max}}_{0,k}(X|D,\beta)]^{\mathrm{vir}} =  \prod_{i=1}^k (-1)^{d_i+1} d_i \cdot [\Kup_{0,0}(\oplus_{i=1}^k \OO_X(-D_i),\beta)]^{\mathrm{vir}}.
\]
The pushforward on the left hand side is suppressed, while the right hand is a naturally a class on the Kontsevich space. 
}

The original form was conjectured by van Garrel--Graber--Ruddat~\cite{vGGR}. Fan--Wu and Tseng--You observed that if $D$ is smooth, the original proof yields the strong form. The general strong form was conjectured by Tseng--You~\cite{TY20}. The conjecture holds in numerical form in many cases~\cite{BBvG,BBvG2}. 

We refer to the local theory class cut down by the divisorial evaluations as appear in the strong form above as the \textit{evaluation local theory} of $(X|D)$. The following observation is elementary. 

\noindent
{\bf Observation.} {\it The evaluation local theory of $(X|D)$ coincides up to sign with the naive theory of $(X|D)$. After pushforward to $\Kup_{0,0}(X,\beta)$, the naive theory coincides up to explicit multiplicity with the local theory of $(X|D)$ as a homology class on $\Kup_{0,0}(X,\beta)$.}

With this observation at hand, we dispense with local Gromov--Witten theory and focus on the more general question of when the logarithmic and naive theories coincide. 

\subsection{Results} Our first result proves that the naive theory does not coincide with the logarithmic theory, giving counterexamples to both forms of the local-logarithmic correspondence. 

\begin{customthm}{X}\label{thm: counterexample}
The naive and logarithmic maximal contact theories do not coincide in degree $2$ for $\PP^2$ equipped with the divisor consisting of $2$ lines. The strong form of the correspondence fails for this geometry in degree $2$. The original form of the correspondence fails in degree $4$.
\end{customthm}

The result is proved by direct geometric analysis. The proof gives a basic flavour of the naive theory and implies that the naive theory is not even enumerative in genus zero for logarithmically convex pairs. The difference between the theories is controlled by the following result, see \S\S \ref{sec: subdivisions}--\ref{sec: blowup formula} and in particular Theorem~\ref{thm: products comparison}.

\begin{customthm}{Y}\label{thm: ample-conjecture}
The difference between the logarithmic and local/naive maximal contact Gromov--Witten invariants of a section pair $(X|D)$ is determined algorithmically in terms of tautological integrals on the moduli space of stable maps to $X$. 
\end{customthm}

Numerical consequences can be extracted. For example, the logarithmic theory of $\PP^2$ relative to two lines can be computed with primary insertions in degree up to $4$, as the corrections will not contribute, see Remark~\ref{rem: deg-4}. A systematic study will appear in future work.

The result is a consequence of the following much stronger result, which implies that the difference between the two theories is captured by Chern classes of tautological bundles, Segre classes of boundary strata in the moduli space of relative maps, and descendent integrals thereon.

Let $\PP$ denote a product of $k$ projective spaces and $H$ a divisor that is a union of hyperplanes $H_1,\ldots, H_k$ pulled back from each factor. 

\begin{customthm}{Z}\label{thm: convex-conjecture}
Let $\Kup^{\mathrm{max}}_{0,k}(\PP|H_i,\beta)$ be the space of logarithmic stable maps to $\PP$ that are maximally tangent to $H_i$ at the $i$th marked point. There is an explicit sequence of weighted blowups of the Kontsevich space 
\[
\Kup_{0,k}(\PP,\beta)^\dagger\to \Kup_{0,k}(\PP,\beta)
\]
along smooth centres such that, denoting strict transforms as $\Kup^{\mathrm{max}}_{0,k}(\PP|H_i,\beta)^\dagger$ and suppressing pushforwards, there is an equality of cycles in the Chow group of $\Kup_{0,k}(\PP,\beta)$:
\[
[\Kup^{\mathrm{max}}_{0,k}(\PP|H_1,\beta)^\dagger]\cdots [\Kup^{\mathrm{max}}_{0,k}(\PP|H_k,\beta)^\dagger] = [\Kup^{\mathrm{max}}_{0,k}(\PP|H,\beta)].
\]
\end{customthm}

Blowups of moduli spaces have appeared in recent work on logarithmic Gromov--Witten theory~\cite{R19,R19b}. The theorem above is considerably stronger. The birational modifications in those papers are not made explicit, while the result above is completely algorithmic, without arbitrary choices. The combinatorics of the maximal contacts situation is leveraged heavily. A reader will find that the combinatorial arguments, manipulating the cone stack of tropical maps, are delicate. These arguments are crucial in deducing structure results for the birational models of stable maps spaces. Outside the maximal contact setup, a sequence of blowups exists but cannot be made explicit. The utility of a general systematic description is likely to be high. In particular, we are not aware of any other methods that calculate the set of invariants that our algorithm calculates. 

Insights from this analysis lead to a new range of cases where the correspondence  holds, see \S\ref{sec: product-section}. These are not covered by the existing literature. 

\begin{customthm}{W}\label{thm: product-conjecture}
Let $X_1,\ldots, X_k$ be smooth, equipped with smooth hyperplane sections $D_1,\ldots, D_k$. The local-logarithmic-naive correspondence holds for the pair $(\prod X_i|\sum D_i)$ with primary factorwise insertions.
\end{customthm}
The condition ``primary factorwise insertions'' is explained in \S \ref{sec: primary factorwise insertions}. It includes in particular all primary invariants with three markings or fewer. These provide the first non-toric examples of the numerical correspondence in dimension larger than $2$. Numerical consequences may again be extracted: invariants of the pair $(\PP^3 \times \PP^2|K3+E)$ where $K3$ and $E$ are a quartic and a cubic, can be computed by~\cite{KlemmP08,Ga03}.

\subsection{Rank reduction and further questions} The local-logarithmic correspondence is one among a number of beautiful results in the relative Gromov--Witten theory of a smooth pair, starting with Gathmann's striking work~\cite{GathmannRelative}. In simple normal crossings geometries, results are much harder to come by and the analogue of Gathmann's recursion is not known. The difficulty of working with the invariants is visible in the degeneration formalism~\cite{ACGS17,R19}.

An approach to our results via degeneration appears to be difficult, at least to the authors. We chart a ``pure thought'' alternative for reducing questions about the geometry of logarithmic stable map spaces to the case of smooth pairs, and implement it completely in the maximal contacts case. The method is restricted to genus zero invariants satisfying a positivity assumption, but even these invariants have not been computed by other methods, not even in principle. Moreover many important phenomena, such as the failure of the local-logarithmic correspondence, appear already in this setting.

Our technique geometrises a categorical insight of Abramovich--Chen~\cite{AbramovichChenLog}. 
Given a logarithmic curve in $(X|D)$, one obtains a logarithmic curve in the smooth pairs $(X|D_i)$ by forgetting the logarithmic structure away from $D_i$. A naive expectation is that the intersection of these loci recovers the locus of logarithmic maps to $(X|D)$. This expectation fails, but is corrected by blowing up the moduli of maps to $X$. The intersection of strict transforms recovers the space of logarithmic maps to $(X|D)$. Tropical geometry informs the blowups used to correct the intersection.

We open two directions for future work. 

\begin{problem}[Moduli factorisation]
For fixed contact order data $\Gamma$, determine an efficient and explicit sequence of blowups at smooth centres $\Kup_{0,k}(\PP,\beta)^\dagger\to \Kup_{0,k}(\PP,\beta)$ such that the strict transform of $\Kup_\Gamma(\PP|H)\to \Kup_{0,k}(\PP,\beta)$ along the blowup is transverse to the strata. 
\end{problem}

A solution would generalise the combinatorics in this paper. It dovetails with the following. For fixed contact orders $\Gamma$ there is a cycle $\Kup_\Gamma(\PP|H)\to \Kup_{0,k}(\PP,\beta)$ in the space of stable maps. For any sufficiently fine logarithmic blowup of the codomain $\Kup_{0,k}(\PP,\beta)^\dagger\to \Kup_{0,k}(\PP,\beta)$, the strict transform of $\Kup_\Gamma(\PP|H)$ is transverse to the boundary of $\Kup_{0,k}(\PP,\beta)^\dagger$. We refer to this as the \textit{transverse class}.

\begin{problem}[The transverse class]
Determine an expression, in terms of tautological classes, for the transverse relative Gromov--Witten class in any sufficiently fine blowup $\Kup_{0,k}(\PP,\beta)^\dagger$ of the moduli space of stable maps.
\end{problem}

We do not extract a closed form in the maximal contacts case; our expression is algorithmic, not closed. A solution to this question would determine the genus zero Gromov--Witten theory of all section pairs, which is beyond the present state of the art, and complete a parallel to~\cite{GathmannRelative}.

To conclude this introduction, we note that an important recent development in the subject has been the development of an approach to Gromov--Witten theory relative to reducible divisors by means of orbifold structures, by Tseng--You~\cite{TY20}. Our naive theory for section pairs coincide with this orbifold theory, see \cite[Corollary 2.2]{BNTY}. Admitting this equality, our framework explains in simple geometric terms why the new orbifold theory does not coincide with the logarthmic theory. 

\subsection*{Comparison with v1}
An earlier version of this paper incorrectly claimed a positive answer to the local-logarithmic conjecture for section pairs. The error was wrongly deducing that the strict transform of the relative Gromov--Witten class was equal to the total transform, which occurred via a misapplication of the vanishing results in~\cite{vGGR}. The technique has been refined in this version, but the basic geometric strategy remains the same.

\subsection*{Acknowledgements} We are grateful to M. van Garrel for encouragement. We have benefited from conversations with D. Abramovich, P. Aluffi, L. Battistella, A. Brini, M. Gross, D. Maulik, S.~Molcho, H. Ruddat, J. Wise, and F. You. We thank the anonymous referee for helpful suggestions. N.N. was supported by EPSRC grant EP/R009325/1 and the Herchel Smith Fund. Figure~\ref{fig example} was created using the package Polymake.

\section{Counterexamples, conics, quartics}\label{sec: counterexamples}
\noindent The counterexamples to the cycle-theoretic correspondence follow the same basic principle. The local theory of a split rank two vector bundle satisfies a simple product rule, coming from the Whitney sum formula for the obstruction bundle. The parallel splitting for the logarithmic theory fails. The analysis here is based on examples computed in the Ph.D. thesis of N.N.~\cite[\S 3]{NabijouThesis}.

\subsection{Unpointed counterexample: plane quartics} Let $H_1,H_2\subseteq \PP^2$ be distinct lines. For $i \in \{1,2\}$ we consider the moduli space
\begin{equation}\label{eqn: deg 4 moduli space} \mathsf K_{0,1}^{\mathrm{max}}(\PP^2|H_i,4)\end{equation}
of logarithmic stable maps to $(\PP^2|H_i)$ with maximal tangency at a single marked point. The logarithmic Euler sequence shows that $T_{(\PP^2|H_i)}$ is convex, so the moduli space \eqref{eqn: deg 4 moduli space} is logarithmically smooth. As such it contains the dimensionally transverse locus, consisting of maps with smooth domain whose image is not contained inside $H_i$, as a dense open. Using an explicit parametrisation of this open locus, we conclude that \eqref{eqn: deg 4 moduli space} is irreducible with dimension and expected dimension equal to $8$. 

Forgetting the logarithmic structures and the marking, we obtain a generically finite map:
\[
\pi^{i}: \mathsf K_{0,1}^{\mathrm{max}}(\PP^2|H_i,4)\to \mathsf K_{0,0}(\PP^2,4).
\]
The target is a smooth Deligne--Mumford stack, of dimension $11$.
\begin{lemma}\label{local-product}
There is an equality of classes in $\mathsf K_{0,0}(\PP^2,4)$:
\[
\pi^{1}_\star [\mathsf K_{0,1}^{\mathrm{max}}(\PP^2|H_1,4)]\cdot \pi^{2}_\star [\mathsf K_{0,1}^{\mathrm{max}}(\PP^2|H_2,4)] =  4^2 \cdot [\Kup_{0,0}( \OO_{\PP^2}(-H_1)\oplus \OO_{\PP^2}(-H_2),4)]^{\mathrm{vir}}.
\]
\end{lemma}
\begin{proof}
The rational Chow groups of $\mathsf K_{0,0}(\PP^2,4)$ possess an intersection product, as this space is smooth. The local class on the right-hand side is the product of the Euler classes of the two local classes associated to the Gromov--Witten theories of the bundles $\OO_{\PP^2}(-H_1)$ and $\OO_{\PP^2}(-H_2)$. By the local-logarithmic correspondence for smooth pairs \cite{vGGR}, each Euler class is equal to the corresponding logarithmic term on the left-hand side. The result follows. 
\end{proof}

Consider the moduli space
\begin{equation*} \mathsf K_{0,2}^{\mathrm{max}}(\PP^2|H_1+H_2,4)\end{equation*}
of logarithmic stable maps with maximal contact to $H_1$ and $H_2$ at markings $x_1$ and $x_2$. As before the logarithmic Euler sequence shows that this space is logarithmically smooth, and contains the locus of maps from smooth domains not mapping into $H_1\cup H_2$ as a dense open. It follows that it is irreducible, with dimension equal to the expected dimension $5$. There is a forgetful morphism:
\begin{equation*} \pi \colon \mathsf K_{0,2}^{\mathrm{max}}(\PP^2|H_1+H_2,4) \to \Kup_{0,0}(\PP^2,4).\end{equation*}
The remainder of this section will focus on the following result, which combined with Lemma~\ref{local-product} demonstrates the failure of the local-logarithmic correspondence.
\begin{proposition}\label{prop: naive class not log class} The following classes in $\Kup_{0,0}(\PP^2,4)$ are not equal:
\begin{equation*} \pi^{1}_\star [\mathsf K_{0,1}^{\mathrm{max}}(\PP^2|H_1,4)]\cdot \pi^{2}_\star [\mathsf K_{0,1}^{\mathrm{max}}(\PP^2|H_2,4)] \neq \pi_\star [\mathsf K_{0,2}^{\mathrm{max}}(\PP^2|H_1+H_2,4)].
\end{equation*}
\end{proposition}
\noindent The forgetful morphisms are all generically injective, so the pushforward classes may be identified with the fundamental classes of the images. Proposition~\ref{prop: naive class not log class} becomes:
\begin{equation}\label{eqn: product not equal image classes} [\pi^{1}(\mathsf K_{0,1}^{\mathrm{max}}(\PP^2|H_1,4))]\cdot [\pi^{2}(\mathsf K_{0,1}^{\mathrm{max}}(\PP^2|H_2,4))]\neq [\pi (\mathsf K_{0,2}^{\mathrm{max}}(\PP^2|H_1+H_2,4))].\end{equation}

\begin{lemma}\label{irred-components-lemma}
The intersection
\begin{equation} \label{eqn: intersection of relative spaces} \pi^{1}(\mathsf K_{0,1}^{\mathrm{max}}(\PP^2|H_1,4)) \cap \pi^{2}(\mathsf K_{0,1}^{\mathrm{max}}(\PP^2|H_2,4))\subseteq  \Kup_{0,0}(\PP^2,4) \end{equation}
contains two irreducible components, each of dimension $5$.\end{lemma}

\begin{proof} The first irreducible component is the main component, which is the closure of the locus of dimensionally transverse maps. This is contained in the intersection \eqref{eqn: intersection of relative spaces}. As noted above, it coincides with the image of the moduli space of logarithmic stable maps to $(\PP^2|H_1+H_2)$:
\begin{equation*} \pi(\Kup^{\max}_{0,2}(\PP^2|H_1+H_2,4)).\end{equation*}
For the second irreducible component, consider the closure in $\mathsf K_{0,0}(\PP^2,4)$ of the locus parameterising maps from rational curves of the form:
\[
C = C_0\cup C_1\cup\cdots\cup C_4\to \PP^2
\]
in which each $C_i$ is smooth, the component $C_0$ is contracted to $H_1\cap H_2$ and meets each of the components $C_1,\ldots,C_4$ in a node, and the remaining components are mapped isomorphically onto lines. The image of such a map is a collection of four lines through the point $H_1 \cap H_2$ and there is an $\Mbar_{0,4}$ moduli for the internal component $C_0$; it follows that this locus is $5$-dimensional. It remains to show that it is contained in each of the images $\pi^i(\mathsf K_{0,1}^{\mathrm{max}}(\PP^2|H_i,4))$.  The image of
\[
\mathsf K_{0,1}^{\mathrm{max}}(\PP^2|H_i,4) \to \mathsf K_{0,1}(\PP^2,4)
\]	
is the closure of its interior; the interior consists of maps from smooth domains which have maximal contact order to $H_i$ but do not map inside $H_i$. The closure is identified by Gathmann's numerical balancing criterion~\cite[Remark~1.7(ii)]{GathmannRelative}. Consider the locus of maps
\[
C = C_0\cup C_1\cup\cdots\cup C_4\to \PP^2
\]
as above, where $C_0$ bears the marked point. Each non-contracted component meets $H_i$ with contact order $1$, and by the numerical criterion we deduce that the locus is contained in the image of $\mathsf K_{0,1}^{\mathrm{max}}(\PP^2|H_i,4)\to\mathsf K_{0,1}(\PP^2,4)$. The claim follows by applying the forgetful morphism $\Kup_{0,1}(\PP^2,4) \to \Kup_{0,0}(\PP^2,4)$.\end{proof}

\begin{proof}[Proof of Proposition \ref{prop: naive class not log class}]
By Lemma~\ref{irred-components-lemma}, the intersection contains at least two irreducible components, each of dimension $5$. Any additional irreducible component must arise as the intersection of images of boundary strata in $\Kup_{0,1}^{\max}(\PP^2|H_i,4)$. We claim its dimension is at most $5$. Consider $\mathsf K_{0,1}^{\mathrm{max}}(\PP^2|H_1,4)$. This has dimension $8$, and any boundary stratum has dimension at most $7$. Forgetting the marking reduces the dimension to $6$ unless the marking lies on a contracted tail. In the former case, as in general, the resulting locus in $\Kup_{0,0}(\PP^2,4)$ does not generically satisfy the numerical criterion with respect to $H_2$, which cuts the dimension to at least $5$. In the latter case, the component of $4$ lines has been dealt with already. The remaining possibility comprises a conic tangent to $H_1$ and two lines through the tangency point. Elementary geometry again bounds the dimension at $5$. It follows that every irreducible component of the intersection \eqref{eqn: intersection of relative spaces} has dimension~$5$.

The left-hand side of \eqref{eqn: product not equal image classes} is the sum of classes of the irreducible components, with positive multiplicities~\cite[Proposition~7.1]{FultonIntersectionTheory}. The space $\mathsf K_{0,0}(\PP^2,4)$ is projective, and Bezout's theorem guarantees that the sum of classes of excess components is not zero in the Chow group of $\mathsf K_{0,0}(\PP^2,4)$. The component $[\pi(\Kup_{0,2}^{\max}(\PP^2|H_1+H_2,4)]$ appears on both sides of \eqref{eqn: product not equal image classes}, so the two sides cannot be equal. \end{proof}

\begin{remark}\label{rem: numerical-counterexample}
The counterexample implies that there is \textit{some} invariant for which the local and logarithmic theories differ. Indeed, Gromov--Witten theory includes all integrals against \textit{tautological} classes on the moduli space of stable maps. Poincar\'e duality furnishes a cohomology class that distinguishes the two classes, but the cohomology of $\mathsf K_{0,0}(\PP^2,4)$ is entirely tautological, see~\cite{Oprea}. An explicit instance is calculated in \S\ref{sec: example}.
\end{remark}

\begin{remark} The intersection in Lemma~\ref{irred-components-lemma} is exactly the union of the two components described above; there are no additional components. This follows from the blowup analysis of \S\S\ref{sec: subdivisions}-\ref{sec: blowup formula}, which also provides a technique to calculate the excess components.\end{remark}

\begin{remark}[Primary correspondence]\label{rem: deg-4}
In the degree $4$ maximal contact geometry for $(\PP^2|H_1+H_2)$, the excess component consists of $4$ lines through a point. As a consequence, this component cannot contribute to a Gromov--Witten invariant with only primary insertions, as the cross ratio of the nodes on the contracted component cannot be fixed by primary insertions. In particular, the local-logarithmic correspondence holds in degrees up to $4$ with primary insertions. 
\end{remark}

\subsection{Pointed counterexample: plane conics}\label{sec: counter-2} The strong form of the correspondence implies the original form, so the counterexample above also falsifies the strong form. We record a simpler failure of the strong form, which occurs in lower degree. We leave some verifications to the reader, since the analysis is simpler than the one above. 

For $i \in \{1,2\}$ consider the moduli space $\Kup^{\max}_{0,2}(\PP^2|H_i,2)$ with maximal contact order at the marking $x_i$ and zero contact order at the marking $x_{\neq i}$. This is the universal curve over the moduli space $\Kup^{\max}_{0,1}(\PP^2|H_i,2)$. Proceeding as above, it suffices to show the following inequality
\begin{equation} \label{eqn: product conic spaces} [\Kup^{\max}_{0,2}(\PP^2|H_1,2)] \cdot [\Kup^{\max}_{0,2}(\PP^2|H_2,2)] \neq [\Kup^{\max}_{0,2}(\PP^2|H_1+H_2,2)]\end{equation}
in $\Kup_{0,2}(\PP^2,2)$, where we have suppressed pushforwards from the notation.

\begin{proof}[Proof of \eqref{eqn: product conic spaces}] We examine the intersection:
\begin{equation}\label{eqn: intersection of conic spaces} \Kup_{0,2}^{\max}(\PP^2|H_1,2) \cap \Kup_{0,2}^{\max}(\PP^2|H_2,2) \subseteq \Kup_{0,2}(\PP^2,2).\end{equation}
This has a main component: the closure of the space of maps intersecting each $H_i$ in precisely one point. This component coincides with the locus $\Kup_{0,2}^{\max}(\PP^2|H_1+H_2,2)$, which has dimension equal to the expected dimension $3$.

A second $3$-dimensional component of the intersection \eqref{eqn: intersection of conic spaces} parametrises maps of the form
\[
C_0\cup C_1\cup C_2\to \PP^2
\]
where $C_0$ bears the two marked points, is contracted to $H_1\cap H_2$, and meets $C_1$ and $C_2$ at distinct points. The components $C_1$ and $C_2$ each map isomorphically onto lines. Elementary geometry shows that this locus has dimension $3$: two dimensions for the two lines and one for the cross ratio of the four points. Direct analysis shows that there are no further irreducible components. This second  irreducible component contributes with positive multiplicity. Therefore the logarithmic class is not the product of the two classes associated to the relative theories of the smooth pairs. This yields a counterexample to the strong form of the conjecture. \end{proof}

\begin{remark}
The examples are the lowest degree failures of the two conjectures. The degree $2$ counterexample above does not yield a counterexample to the original form of the correspondence; the cross ratio of the points in the contracted component is lost, so the class vanishes in the pushforward. So the strong form of the conjecture is genuinely stronger than the original one!
\end{remark}

\section{Correcting the correspondence I: subdivisions and modifications}\label{sec: subdivisions}
\noindent The failure of the local-logarithmic correspondence stems from the fact that moduli spaces of logarithmic maps do not satisfy a naive product formula over the space of ordinary maps:
\begin{equation*} \Kup(X|D_1) \times_{\Kup(X)} \Kup(X|D_2) \neq \Kup(X|D).\end{equation*}
The left-hand side can include excess components, even in convex settings where the right-hand side is irreducible. The local and naive theories do satisfy a product formula, so the local-logarithmic correspondence cannot hold in generality. This observation led to the counterexamples of~\S \ref{sec: counterexamples}.

In the next two sections, we establish a method for calculating the defect between the naive and logarithmic theories. We transversalise the naive intersection by performing blowups on $\Kup(X)$, and apply Fulton's blowup formula to quantify the difference between the theories.

\subsection{Setup: target geometry and moduli spaces} \label{sec: setup} Consider a target $(X|D)$ with $X=\PP^{n_1}\times \PP^{n_2}$ and $D=D_1+D_2$ a divisor, with each smooth component $D_i$ the pullback of a hyperplane in $\PP^{n_i}$. Spaces of genus zero logarithmic stable maps to $X$, $(X|D_1)$, $(X|D_2)$ and $(X|D)$ are logarithmically unobstructed; the discussion which follows applies to any target satisfying this.

We establish a corrected local-logarithmic correspondence in this setting; the case with more divisor components follows \emph{mutatis mutandis} by replacing $D_2$ with $D_2+\ldots+D_k$, and the case of hyperplane sections follows by virtual pull-back (see \S \ref{sec: virtual pullback}).

We begin by establishing notation for the maximal contacts theory. We fix a curve class $\beta$ and introduce markings $x_1,x_2$ which have maximal tangency with respect to $D_1,D_2$ respectively. We obtain a moduli space of logarithmic stable maps to $(X|D)$ with maximal contacts
\begin{equation*} \Kup_{0,2}^{\max}(X|D,\beta) \end{equation*}
and for $i \in \{1,2\}$ a moduli space of logarithmic stable maps $\Kup_{0,2}^{\max}(X|D_i,\beta)$ to $(X|D_i)$.
The latter space is also two-pointed; the marking $x_{\neq i}$ carries no tangency condition. This is the universal curve over the one-pointed space $\Kup_{0,1}^{\max}(X|D_i,\beta)$.

The following target diagram is cartesian in fine and saturated logarithmic schemes:
\[
\begin{tikzcd}
(X|D)\arrow{r}\arrow{d}\ar[rd,phantom,"\square"] & (X|D_1) \arrow{d} \\
(X|D_2) \arrow{r} & X.
\end{tikzcd}
\]
The moduli spaces of two-pointed logarithmic stable maps enjoy a similar relationship
\[
\begin{tikzcd}
\mathsf K_{0,2}^{\max}(X|D,\beta)\arrow{r}\arrow{d}\ar[rd,phantom,"\square"] & \mathsf K_{0,2}^{\max}(X|D_1,\beta) \arrow{d} \\
\mathsf K_{0,2}^{\max}(X|D_2,\beta) \arrow{r} & \mathsf K_{0,2}(X,\beta)
\end{tikzcd}
\]
see~\cite[Theorem 2.6]{AbramovichChenLog}. This diagram is cartesian in the category of fine and saturated logarithmic stacks, but not typically cartesian in the category of ordinary stacks (the cartesian product in the category of ordinary stacks is instead the naive space). The failure is accounted for by the fact that neither of the morphisms $\mathsf K_{0,2}^{\max}(X|D_i,\beta)\to \mathsf K_{0,2}(X,\beta)$ is integral and saturated.

\subsection{Semistable reduction}\label{sec: weak-ss-red} 
Our strategy is to correct this fibre product, by replacing the morphism $\mathsf{K}_{0,2}^{\max}(X|D_1,\beta) \to \mathsf{K}_{0,2}(X,\beta)$ with an integral and saturated birational model. This will be constructed using weak semistable reduction \cite{AK00,Mol}.

A toroidal morphism $X\to B$ of toroidal embeddings is logarithmically smooth with the divisorial structure. The morphism need not be equidimensional or have reduced fibres. In their work on weak semistable reduction, Abramovich--Karu identify criteria for these properties. 

\begin{lemma}[{\cite[Lemma~4.1]{AK00}}]\label{lem: comb-equi}\label{lem: AK surject cone}
Let $f: X\to B$ be a toroidal morphism of toroidal embeddings and let $\Sigma_X\to\Sigma_B$ be the morphism of cone complexes. Then $f$ has equidimensional fibres if and only if every cone of $\Sigma_X$ surjects onto a cone of $\Sigma_B$. 
\end{lemma}

\begin{lemma}[{\cite[Lemma~5.2]{AK00}}]\label{lem: AK reduced}
Let $f: X\to B$ be a toroidal morphism with equidimensional fibres and let $\Sigma_X\to\Sigma_B$ be the morphism of cone complexes. Then $f$ has reduced fibres if and only if for every cone $\sigma$ with image cone $\tau$, the image of the morphism on associated lattices is saturated.
\end{lemma}

A toroidal morphism satisfying the conditions in the lemmas above is \textbf{weakly semistable}. Any toroidal morphism can be modified to a weakly semistable one.

\begin{proposition}[Toroidal weak semistable reduction{~\cite{AK00}}]
Let $f: X\to B$ and $\Sigma_f: \Sigma_X\to\Sigma_B$ be as above. There exist subdivisions of the source and target $\Sigma^\dagger_X\to\Sigma^\dagger_B$, such that the resulting morphism $X^\dagger\to B^\dagger$ is equidimensional. By applying a sequence of root stack constructions (change of lattice) to $B^\dagger$, we obtain a Deligne--Mumford stack $\mathcal{B}^\dagger$ and a new morphism $X^\dagger\to \mathcal{B}^\dagger$
which is equidimensional with reduced fibres. 
\end{proposition}
The Abramovich--Karu construction is non-unique, depending on an auxiliary choice of piecewise-linear support functions. Later work of Molcho \cite[Theorem 2.4.2]{Mol} shows that if the morphism $\Sigma_f$ is proper and surjective, there is a unique minimal choice. The construction declares the image of every cone to be a cone, and subdivides the intersections as necessary.

\subsection{The subdivision} \label{sec: construction of subdivisions}
We apply the construction in the previous section to spaces of logarithmic stable maps. The first step is to replace each of the moduli spaces $\Kup^{\max}_{0,2}(X|D_i,\beta)$ for $i=1,2$ with Kim's space of logarithmic stable maps to expansions \cite{Kim08}. As discussed in \cite[Section 2.1]{BNR1}, this is a logarithmic modification of the Abramovich--Chen--Gross--Siebert moduli space, representing the subfunctor of \textbf{image-ordered} tropical maps. The tropicalisation
\begin{equation*} \Tup_{0,2}^{\max}(X|D_i,\beta) = \mathsf{Trop\, } \Kup_{0,2}^{\max}(X|D_i,\beta)  \end{equation*}
is the cone complex parametrising image-ordered degree-weighted tropical stable maps to $\RR_{\geq 0}$. The space $\Kup_{0,2}(X,\beta)$ has logarithmic structure induced by its normal crossings boundary (equivalently, by viewing it as a space of logarithmic stable maps to a trivial logarithmic scheme), and the tropicalisation
\begin{equation*} \Tup_{0,2}(X,\beta) = \mathsf{Trop\, } \Kup_{0,2}(X,\beta) \end{equation*}
is the cone complex parametrising degree-weighted tropical stable curves. We apply weak semistable reduction to the morphism
\begin{equation*} \Kup_{0,2}^{\max}(X|D_1,\beta) \to \Kup_{0,2}(X,\beta).\end{equation*}
This produces subdivisions of the associated cone complexes with an induced morphism
\[
\mathsf T_{0,2}^{\max}(X|D_1,\beta)^\dagger\to \mathsf T_{0,2}(X,\beta)^\dagger
\]
which is combinatorially equidimensional and reduced; it satisfies the polyhedral criteria for these conditions. On the associated logarithmic modifications, we obtain a morphism
\[
\mathsf K_{0,2}^{\max}(X|D_1,\beta)^\dagger \to \mathsf{K}_{0,2}(X,\beta)^\dagger
\]
which is integral and saturated. We subdivide $\Tup_{0,2}(X|D_2,\beta)$ by pulling back the subdivision $\Tup_{0,2}(X,\beta)^\dag$ of $\Tup_{0,2}(X,\beta)$ (note the asymmetry between $D_1$ and $D_2$ in this construction). We thus obtain a diagram
\[
\begin{tikzcd}
\mathsf K_{0,2}^{\max}(X|D,\beta)^\dag \arrow{r}\arrow{d}\ar[rd,phantom,"\square"] & \mathsf K_{0,2}^{\max}(X|D_1,\beta)^\dag \ar[d,"g"] \\
\mathsf K_{0,2}^{\max}(X|D_2,\beta)^\dag \arrow{r} & \mathsf K_{0,2}(X,\beta)^\dag
\end{tikzcd}
\]
which, since the morphism $g$ is now integral and saturated, is cartesian in both the category of fine and saturated logarithmic stacks and the category of ordinary stacks. The fibre product $\Kup_{0,2}^{\max}(X|D,\beta)^\dag$ is a birational model for $\Kup_{0,2}^{\max}(X|D,\beta)$.

\begin{remark} The construction is canonical, since the morphism of cone complexes
\begin{equation*} \Tup_{0,2}^{\max}(X|D_1,\beta) \to \Tup_{0,2}(X,\beta) \end{equation*}
is surjective. This is shown at the start of the next section. \end{remark}

\begin{remark} The preceding subdivisions do not require a change of lattice (saturation), as the minimal monoid associated to a tropical stable map is automatically saturated over the minimal monoid associated to the underlying tropical curve \cite[\S 1.5]{GS13}. However, the interpretation of our weighted blowups as stacky modifications will require a stacky change of lattice, see \S \ref{sec: weighted blowups}.\end{remark}

\subsection{Modular description: image-ordering (left-to-right)} \label{sec: modular description image ordering} In order to access the intersection theory of these modifications, it is necessary to obtain a more explicit description of the subdivisions involved. We begin with a modular interpretation for the subdivision $\Tup_{0,2}(X,\beta)^\dag$ in terms of order relations on the vertices of the tropical curve. A similar discussion can be found in~\cite{CMR14b} and additional examples are discussed there.

\begin{remark} Both the results and the arguments of this section apply beyond the maximal contacts setting, to any moduli space of genus zero logarithmic stable maps relative to a smooth divisor. \end{remark}

Given a two-pointed, degree-weighted stable tropical curve $\sqC$ over a base cone $\sigma$, we may assign the \textbf{formal expansion factor} $D_1 \cdot \beta$ to the semi-infinite leg corresponding to $x_1$, and the formal expansion factor $0$ to the semi-infinite leg corresponding to $x_2$. Having done this, there is then a unique way to assign a formal expansion factor $m_e$ to each (directed) edge $e \in \sqC$, in such a way that the resulting tropical curve is balanced; this is a consequence of the genus zero hypothesis. From this, we obtain a tropical map
\begin{equation*} f \colon \sqC \to \mathrm{Hom}(\sigma,\RR),\end{equation*}
well-defined up to overall translation in $\RR$. For vertices $v_1,v_2 \in \sqC$ we declare
\begin{equation*}  f(v_1) \leq f(v_2) \text{\quad if and only if \quad} f(v_2) - f(v_1) \in \Hom(\sigma,\RR_{\geq 0})\end{equation*}
and observe that this defines a partial ordering on the vertices of $\sqC$.

\begin{proposition}  \label{prop: modular description image ordered} $\Tup_{0,2}(X,\beta)^\dag$ is the space of degree-weighted stable tropical curves, such that \textbf{the $f(v)$ are totally ordered.} The cones of this subdivision are the images of cones of $\Tup_{0,2}^{\max}(X|D_1,\beta)$.	
\end{proposition}

We refer to $\Tup_{0,2}(X,\beta)^\dag$ as the moduli space of \textbf{image-ordered} tropical curves. A similar construction was outlined in~\cite[Section 2.1]{BNR1}.

\begin{proof} Temporarily denote the moduli space of image-ordered tropical curves by $\Tup_{0,2}(X,\beta)^\ddagger$; it is clear that this is a subdivision of $\Tup_{0,2}(X,\beta)$.

Consider a cone $\tau \in \Tup^{\max}_{0,2}(X|D_1,\beta)$. This corresponds to a combinatorial type of tropical stable map to $\RR_{\geq 0}$, and if we consider the image $\bar{\tau} \subseteq |\Tup_{0,2}(X,\beta)|$ and restrict the universal curve $\sqC$ to $\bar{\tau}$, we obtain a tropical curve whose $f(v)$ are totally ordered (we obtain a total ordering because we work with Kim's space). This total ordering determines a combinatorial type of image-ordered curve, corresponding to a cone $\rho \in \Tup_{0,2}(X,\beta)^\ddagger$ such that $\bar{\tau} \subseteq \rho$. We need to show that in fact $\bar{\tau} = \rho$.

The cone $\tau$ is simplicial, with coordinates over $\QQ$ given by the target edge lengths $l_1,\ldots,l_k$. We may assume that over $\tau$ there is at least one vertex $v_0 \in \sqC$ mapping to $0 \in \RR_{\geq 0}$ (if not, replace $\tau$ with the subcone defined by $l_1=0$, and note that this does not alter $\bar\tau$ or $\rho$).

Choosing for each $i$ a stable vertex $v_i \in \sqC$ mapping to the $i$th target vertex, we have $f(v_0) < f(v_1) < \ldots < f(v_k)$ on $\rho$, and every other vertex satisfies $f(v)=f(v_i)$ for some $i$. Thus, we see that $\rho$ is also a simplicial cone, with coordinates over $\QQ$ given by:
\begin{equation*} f(v_1)-f(v_0), f(v_2) - f(v_1), \ldots, f(v_k) - f(v_{k-1}).\end{equation*}
The map $\tau \to \rho$ is given by  $l_i \mapsto f(v_i) - f(v_{i-1})$, which is clearly surjective, so $\bar{\tau} = \rho$ as required.

On the other hand, given a cone $\rho \in \Tup_{0,2}(X,\beta)^\ddagger$ corresponding to a combinatorial type of image-ordered tropical curve, we obtain a unique minimal combinatorial type for a stable tropical map $\sqC \to \RR_{\geq 0}$, by forcing vertices in $\sqC$ with minimal $f(v)$ to map to $0 \in \RR_{\geq 0}$; there are no further edge length relations as $\sqC$ has genus zero. This corresponds to a cone $\tau \in \Tup^{\max}_{0,2}(X|D_1,\beta)$ and it follows from the discussion above that $\bar{\tau} = \rho$. \end{proof}

\begin{corollary}\label{cor: D1 space not modified} The subdivision procedure does not modify the source cone complex:
\begin{equation*}\Tup_{0,2}^{\max}(X|D_1,\beta)^\dag =\Tup_{0,2}^{\max}(X|D_1,\beta).\end{equation*}\end{corollary}
\begin{proof} The image $\bar{\tau}$ of every cone of $\Tup_{0,2}^{\max}(X|D_1,\beta)$ is a cone in the image-ordered subdivision, so the images of two cones cannot intersect away from the image of a common face.\end{proof}

\begin{remark} The previous result helps in understanding the subdivision procedure. It is essentially unimportant for our later arguments. \end{remark}

\begin{corollary} \label{cor: iterated blowup smooth} The subdivision $\Tup_{0,2}(X,\beta)^\dag$ is simplicial and $\Kup_{0,2}(X,\beta)^\dag$ is a smooth orbifold.\end{corollary}
\begin{proof} We saw in the proof of Proposition \ref{prop: modular description image ordered} that image-ordered cones are simplicial. The fact that $\Kup_{0,2}(X,\beta)^\dag$ is smooth follows immediately, interpreting the logarithmic modification as a non-representable orbitoroidal embedding (see \S \ref{sec: weighted blowups} for details). \end{proof}

\subsection{Modular description: alignment (right-to-left)} The results of the previous subsection are general, applying to moduli spaces with arbitrary tangency orders. When the contact order is maximal, we exhibit a combinatorial factorisation of the subdivision, describing it as a sequence of weighted stellar subdivisions along smooth cones. The description resembles the radial alignments in~\cite{RSW17A}.

Observe that if we let $v_0 \in \sqC$ denote the vertex containing the marking $x_1$, the balancing condition implies that $f(v_0)$ must be maximal amongst the $f(v)$. If we therefore let
\begin{equation*} \varphi(v) = f(v_0) - f(v) \in \Hom(\sigma,\RR_{\geq 0}) \end{equation*}
then we see that totally ordering the $f(v)$ is equivalent to totally ordering the $\varphi(v)$. We think of $\varphi(v)$ as the \textbf{distance from the root $v_0$}: it is the expansion-factor-weighted sum of the edge lengths along the unique path connecting $v_0$ to $v$. We obtain:

\begin{proposition}\label{prop: modular description distance from root} $\Tup_{0,2}(X,\beta)^\dagger$ is the moduli space of degree-weighted stable tropical curves, such that the distances $\varphi(v)$ from the root $v_0$ are totally ordered.\end{proposition} 
We call such a tropical curve \textbf{radially aligned}, or \textbf{aligned}, with respect to $v_0$.
 
\subsection{Iterative description} \label{sec: iterative description} The modular interpretation via alignments gives a very concrete iterative description of $\Tup_{0,2}(X,\beta)^\dagger \to \Tup_{0,2}(X,\beta)$, and therefore of the logarithmic modification $\Kup_{0,2}(X,\beta)^\dagger \to \Kup_{0,2}(X,\beta)$. This description, inspired by results in~\cite{RSW17A,VZ08}  will be crucial in \S \ref{sec: blowup formula}.

\begin{definition}
A \textbf{floral cone} $\sigma \in \Tup_{0,2}(X,\beta)$ is a cone indexed by a type of the following form:
\begin{center}
\begin{tikzpicture}


	\draw[fill=black] (0,0) circle[radius=2pt];
	\draw (0,0) node[left]{\small$\beta_1$};
	\draw[fill=black] (0,0) -- (1,-2);
		
	\draw (1,0) node{$\ldots$};
	
	\draw[fill=black] (2,0) circle[radius=2pt];
	\draw (2,0) node[right]{\small$\beta_r$};
	\draw[fill=black] (2,0) -- (1,-2);
	
	\draw[fill=black] (1,-2) circle[radius=2pt];
	\draw (1,-2) node[left]{\small$\beta_0$};
	\draw[->] (1,-2) -- (1,-2.5);
	\draw (1.05,-2.45) node[below]{\small$x_1$};
	
	
	
	
\end{tikzpicture}	
\end{center}	
\end{definition}
\noindent The vertex supporting the marking $x_2$ is allowed to be arbitrary, and is denoted $v(x_2)$. We impose a partial ordering on the floral cones, as follows:
\begin{equation}\label{order floral cones}
\begin{tikzpicture}[baseline=(current  bounding  box.center)]
	\draw[fill=black] (0,0) circle[radius=2pt];
	\draw (0,0) node[left]{\small$\beta_1$};
	\draw[fill=black] (0,0) -- (1,-2);
	\draw (1,0) node{$\ldots$};
	\draw[fill=black] (2,0) circle[radius=2pt];
	\draw (2,0) node[right]{\small$\beta_r$};
	\draw[fill=black] (2,0) -- (1,-2);
	\draw[fill=black] (1,-2) circle[radius=2pt];
	\draw (1,-2) node[left]{\small$\beta_0$};
	\draw[->] (1,-2) -- (1,-2.5);
	\draw (1.05,-2.45) node[below]{\small$x_1$};
	
	\draw (3,-1.3) node[right]{$<$};
	
	\draw[fill=black] (4.6,0) circle[radius=2pt];
	\draw (4.6,0) node[left]{\small$\beta_1^\prime$};
	\draw[fill=black] (4.6,0) -- (5.6,-2);
	\draw (5.6,0) node{$\ldots$};
	\draw[fill=black] (6.6,0) circle[radius=2pt];
	\draw (6.6,0) node[right]{\small$\beta_{r^\prime}^\prime$};
	\draw[fill=black] (6.6,0) -- (5.6,-2);
	\draw[fill=black] (5.6,-2) circle[radius=2pt];
	\draw (5.6,-2) node[left]{\small$\beta_0^\prime$};
	\draw[->] (5.6,-2) -- (5.6,-2.5);
	\draw (5.65,-2.45) node[below]{\small$x_1$};
	
		
\end{tikzpicture}
\end{equation}
if and only if
\begin{enumerate}
\item $\beta_0 < \beta_0^\prime$; or\medskip
\item $\beta_0=\beta_0^\prime$ and $r < r^\prime$; or\medskip
\item $\beta_0=\beta_0^\prime,r=r^\prime$ and $v(x_2) \neq v_0$ but $v^\prime(x_2)=v_0$.	
\end{enumerate}

\begin{remark} \label{rmk: cone not in star} Given floral cones $\sigma,\sigma^\prime \in \Tup_{0,2}(X,\beta)$ with $\sigma^\prime \in \text{Star}(\sigma)$, stability ensures that $\sigma^\prime < \sigma$. Equivalently:
\begin{equation*} \sigma^\prime  \not < \sigma \Rightarrow \sigma^\prime \not\in \text{Star}(\sigma). \end{equation*}
Therefore $\sigma^\prime$ will be unaffected by taking a weighted stellar subdivision along $\sigma$, i.e. it will remain a cone in the subdivided cone complex.
\end{remark}

Given this setup, we have the following strong combinatorial structure result:

\begin{theorem}\label{thm: iterative description}
 The morphism $\Tup_{0,2}(X,\beta)^\dagger \to \Tup_{0,2}(X,\beta)$ is an iterated weighted stellar subdivision of $\Tup_{0,2}(X,\beta)$ along floral cones, in an order extending the partial order \eqref{order floral cones}.
\end{theorem}
A \textbf{floral stratum} is a closed boundary stratum $Z(\sigma) \subseteq \Kup_{0,2}(X,\beta)$ corresponding to a floral cone $\sigma \in \Tup_{0,2}(X,\beta)$.

\begin{corollary} \label{cor: iterative blowup} The morphism $\Kup_{0,2}(X,\beta)^\dagger \to \Kup_{0,2}(X,\beta)$ is an iterated weighted blowup of $\Kup_{0,2}(X,\beta)$ along strict transforms of floral strata, in an order extending the partial order \eqref{order floral cones}.\end{corollary}

\begin{remark} \label{rmk: strict transforms well defined} The statement of Remark \ref{rmk: cone not in star} asserts that $Z(\sigma^\prime)$ is not contained in the blowup centre $Z(\sigma)$, and therefore that its strict transform under the blowup --- indexed by the same cone $\sigma^\prime$ in the subdivided complex --- is nonempty.
\end{remark}

\begin{proof}[Proof of Theorem \ref{thm: iterative description}] Fix a cone $\sigma \in \Tup_{0,2}(X,\beta)$ corresponding to a combinatorial type of a two-pointed, degree-weighted tropical curve $\sqC$. As before, let $v_0 \in \sqC$ denote the root vertex containing the marking $x_1$, and use the balancing condition to assign formal expansion factors to every edge.
	
We construct the radially aligned subdivision $\sigma^\dag \to \sigma$ inductively. The idea is as follows: in order to choose a total ordering of the distances $\varphi(v)$ of the vertices from the root, we first must decide which vertex has smallest $\varphi(v)$. Having done this, we then need to decide which vertex is the next-smallest, i.e. the smallest amongst the remaining vertices, and so on. Each step is a weighted stellar subdivision of a floral cone, consistent with the ordering \eqref{order floral cones}.
	
Edges with zero expansion factor play no role in the subdivision, since their length parameters do not appear in the $\varphi(v)$. Therefore, we formally contract all such edges for this discussion. Orient the graph $\sqC$ in such a way that every edge points away from the root. The first step is to decide which $v$ has minimal $\varphi(v)$; the candidate vertices are the immediate descendants of $v_0$: 
\begin{center}
\begin{tikzpicture}[scale=0.8]
	\draw[fill=black] (0,0) circle[radius=2pt];
	\draw (0,0) node[left]{\small$v_1$};
	\draw[fill=black] (0,0) -- (1,-2);
	\draw (0.5,-1) node[left]{\small$\varphi(v_1)$};
		
	\draw (0,0) -- (-0.25,0.5);
	\draw[fill=black] (-0.3,0.6) circle[radius=0.5pt];
	\draw[fill=black] (-0.35,0.7) circle[radius=0.5pt];
	\draw (0,0) -- (0.25,0.5);
	\draw[fill=black] (0.3,0.6) circle[radius=0.5pt];
	\draw[fill=black] (0.35,0.7) circle[radius=0.5pt];

	\draw (1,0) node{$\ldots$};
	
	\draw[fill=black] (2,0) circle[radius=2pt];
	\draw (2,0) node[right]{\small$v_r$};
	\draw[fill=black] (2,0) -- (1,-2);
	\draw (1.5,-1) node[right]{\small$\varphi(v_r)$};
	
	\draw (2,0) -- (1.75,0.5);
	\draw[fill=black] (1.7,0.6) circle[radius=0.5pt];
	\draw[fill=black] (1.65,0.7) circle[radius=0.5pt];
	\draw (2,0) -- (2.25,0.5);
	\draw[fill=black] (2.3,0.6) circle[radius=0.5pt];
	\draw[fill=black] (2.35,0.7) circle[radius=0.5pt];
	
	\draw[fill=black] (1,-2) circle[radius=2pt];
	\draw (1,-2) node[left]{\small$v_0$};
	\draw[->] (1,-2) -- (1,-2.5);
	\draw (1.05,-2.45) node[below]{\small$x_1$};
\end{tikzpicture}	
\end{center}	
Setting all coordinates other than $\varphi(v_1),\ldots,\varphi(v_r)$ to zero, we obtain the floral subcone:
\begin{center}
\begin{tikzpicture}[scale=0.8]
	\draw[fill=black] (0,0) circle[radius=2pt];
	\draw (0,0) node[left]{\small$v_1$};
	\draw[fill=black] (0,0) -- (1,-2);
	\draw (0.5,-1) node[left]{\small$\varphi(v_1)$};
		
	\draw (1,0) node{$\ldots$};
	
	\draw[fill=black] (2,0) circle[radius=2pt];
	\draw (2,0) node[right]{\small$v_r$};
	\draw[fill=black] (2,0) -- (1,-2);
	\draw (1.5,-1) node[right]{\small$\varphi(v_r)$};
	
	\draw[fill=black] (1,-2) circle[radius=2pt];
	\draw (1,-2) node[left]{\small$v_0$};
	\draw[->] (1,-2) -- (1,-2.5);
	\draw (1.05,-2.45) node[below]{\small$x_1$};
\end{tikzpicture}	
\end{center}	
The weighted stellar subdivision of $\sigma$ along this floral subcone subdivides $\sigma$ into cones, on each of which we have
\begin{equation*} \text{min}(\varphi(v_1),\ldots,\varphi(v_n)) = \varphi(v_i) \end{equation*}
for some $i$. The weights are determined by the edge expansion factors, noting that the $\varphi(v_i)$ may not be primitive in $\sigma$. On each cone of the subdivision there is a minimal vertex of $\sqC$. This forms the base of the induction.
	
For the induction step, choose a cone $\rho$ of the subdivision constructed so far; to simplify notation, we assume that $\rho$ is maximal. On this cone we have a total ordering of a subset $\{ u_1,\ldots,u_k\}$ of the vertices of $\sqC$:
\begin{equation} \label{eqn: order so far} \varphi(u_1) < \ldots < \varphi(u_k) < \varphi(v) \ \ \text{ for\,\, } v \not\in \{ u_1,\ldots,u_k,v_0\}.
	\end{equation}
Suppose that at the previous step we had taken a weighted stellar subdivision along a floral cone
\begin{equation}\label{first floral cone}
\begin{tikzpicture}[baseline=(current  bounding  box.center),scale=0.8]
	\draw[fill=black] (0,0) circle[radius=2pt];
	\draw (0,0) node[left]{\small$v_1$};
	\draw (0,0) node[above]{\small$\beta_1$};
	\draw[fill=black] (0,0) -- (1,-2);
	\draw (0.5,-1) node[left]{\small$\varphi(v_1)$};
		
	\draw (1,0) node{$\ldots$};
	
	\draw[fill=black] (2,0) circle[radius=2pt];
	\draw (2,0) node[right]{\small$v_r$};
	\draw (2,0) node[above]{\small$\beta_r$};
	\draw[fill=black] (2,0) -- (1,-2);
	\draw (1.5,-1) node[right]{\small$\varphi(v_r)$};
	
	\draw[fill=black] (1,-2) circle[radius=2pt];
	\draw (1,-2) node[left]{\small$v_0$};
	\draw (1,-2) node[right]{\small$\beta_0$};
	\draw[->] (1,-2) -- (1,-2.5);
	\draw (1.05,-2.45) node[below]{\small$x_1$};
\end{tikzpicture}	
\end{equation}
and that (without loss of generality) $\rho$ is the cone of this subdivision on which $\varphi(v_1)$ is minimal (that is, $v_1=u_k$ in \eqref{eqn: order so far}).

The candidates for the next-smallest vertex of $\sqC$ comprise the vertices $v_2,\ldots,v_r$ along with any immediate descendants of $v_1$; denote these by $w_1,\ldots,w_s$. The following picture describes $\sqC$:
\begin{center}
\begin{tikzpicture}[scale=0.8]
\draw[fill] (0,0) circle[radius=2pt];
\draw (0,0) node[left]{\small$v_0$};
\draw[->] (0,0) -- (0,-0.5);
\draw (0.05,-0.5) node[below]{\small$x_1$};

\draw (0,0) -- (-0.5,0.5);
\draw[fill] (-0.6,0.6) circle[radius=0.5pt];
\draw[fill] (-0.7,0.7) circle[radius=0.5pt];
\draw (0,0) -- (0.5,0.5);
\draw[fill] (0.6,0.6) circle[radius=0.5pt];
\draw[fill] (0.7,0.7) circle[radius=0.5pt];
\draw (0,0.5) node[]{$\ldots$};

\draw[fill] (-2,2) circle[radius=2pt];
\draw (-2,2) node[left]{\small$v_1$};
\draw (-2,2) -- (-1.5,1.5);
\draw[fill] (-1.4,1.4) circle[radius=0.5pt];
\draw[fill] (-1.3,1.3) circle[radius=0.5pt];


\draw (-2,2) -- (-2.75,3);
\draw[fill] (-2.75,3) circle[radius=2pt];
\draw (-2.75,3) node[left]{\small$w_1$};
\draw[fill] (-2.75,3.2) circle[radius=0.5pt];
\draw[fill] (-2.75,3.3) circle[radius=0.5pt];

\draw (-2,3) node[]{$\ldots$};

\draw (-2,2) -- (-1.25,3);
\draw[fill] (-1.25,3) circle[radius=2pt];
\draw (-1.25,3) node[right]{\small$w_s$};
\draw[fill] (-1.25,3.2) circle[radius=0.5pt];
\draw[fill] (-1.25,3.3) circle[radius=0.5pt];

\draw[fill] (0,2) circle[radius=2pt];
\draw[fill] (0,2.2) circle[radius=0.5pt];
\draw[fill] (0,2.3) circle[radius=0.5pt];
\draw (0,2) node[left]{\small$v_2$};
\draw (0,2) -- (0,1.5);
\draw[fill] (0,1.4) circle[radius=0.5pt];
\draw[fill] (0,1.3) circle[radius=0.5pt];

\draw (1,2) node[]{$\ldots$};

\draw[fill] (2,2) circle[radius=2pt];
\draw[fill] (2,2.2) circle[radius=0.5pt];
\draw[fill] (2,2.3) circle[radius=0.5pt];
\draw (2,2) node[right]{\small$v_r$};
\draw (2,2) -- (1.5,1.5);
\draw[fill] (1.4,1.4) circle[radius=0.5pt];
\draw[fill] (1.3,1.3) circle[radius=0.5pt];

\end{tikzpicture}	
\end{center}	
The following functions then form part of a coordinate system for the cone $\rho$
\begin{equation}\label{params} \varphi(v_2) - \varphi(v_1), \ldots, \varphi(v_r) - \varphi(v_1), \varphi(w_1) - \varphi(v_1), \ldots, \varphi(w_s) - \varphi(v_1)\end{equation}
(geometrically, the curve is destabilised by slicing it with the circle of radius $\varphi(v_1)$, and the parameter $\varphi(v_i)-\varphi(v_1)$ is the length of the final edge segment preceding the vertex $v_i$). Set all parameters other than \eqref{params} to zero to obtain the following floral subcone:
\begin{equation}\label{second floral cone}
\begin{tikzpicture}[baseline=(current  bounding  box.center),scale=0.8]
\draw[fill] (0,0) circle[radius=2pt];
\draw (0,0) node[left]{\small$v_0$};
\draw (0.1,0) node[right]{\small$\beta_0+\beta_1$};
\draw[->] (0,0) -- (0,-0.5);
\draw (0.05,-0.5) node[below]{\small$x_1$};

\draw (0,0) -- (-2,1.5);
\draw[fill] (-2,1.5) circle[radius=2pt];
\draw (-2,1.5) node[left]{\small$w_1$};

\draw (-1.5,1.5) node[]{\small$\ldots$};

\draw (0,0) -- (-0.5,1.5);
\draw[fill] (-0.5,1.5) circle[radius=2pt];
\draw (-0.5,1.5) node[left]{\small$w_s$};

\draw (0,0) -- (0.5,1.5);
\draw[fill] (0.5,1.5) circle[radius=2pt];
\draw (0.5,1.5) node[right]{\small$v_2$};

\draw (1.5,1.5) node[]{\small$\ldots$};

\draw (0,0) -- (2,1.5);
\draw[fill] (2,1.5) circle[radius=2pt];
\draw (2,1.5) node[right]{\small$v_r$};
\end{tikzpicture}
\end{equation}
Note that, since $\varphi(v_1)=0$ on this cone, the degree of the root changes from $\beta_0$ in \eqref{first floral cone} to $\beta_0+\beta_1$ in \eqref{second floral cone}. Taking the weighted stellar subdivision along this cone corresponds to choosing a minimum amongst the parameters \eqref{params}. But of course this is equivalent to choosing a minimum amongst:
\begin{equation*}\varphi(v_2), \ldots, \varphi(v_r), \varphi(w_1), \ldots, \varphi(w_s).\end{equation*}
We have completed the induction step of the construction. Either $\beta_1 > 0$, in which case $\beta_0 < \beta_0+\beta_1$. Otherwise $\beta_1=0$, and so by stability we have either $s \geq 2$ and so $r < s + r-1$, or $s=1$ in which case $v_1$ must contain the marking $x_2$, which lies on the vertex $v_0$ once we set $\varphi(v_1)=0$. In every case, we see that the floral locus \eqref{second floral cone} appears strictly later than \eqref{first floral cone} in our ordering \eqref{order floral cones}.
\end{proof}

\begin{remark} \label{rmk: not all floral cones} We note $\Tup_{0,2}(X,\beta)^\dag \to \Tup_{0,2}(X,\beta)$ is not a weighted stellar subdivision along \emph{every} floral cone. It is obtained by subdividing along those  floral cones for which each edge of the tropical curve is assigned a non-zero formal expansion factor with respect to $D_1$; equivalently, those floral cones for which $D_1 \cdot \beta_i > 0$ for $i \in \{1,\ldots,r\}$. This follows from the formal contraction of edges with zero expansion factor, carried out in the proof above.\end{remark}

\begin{example} \label{example subdivision} Take $X=\PP^n$ with $D_1=H_1$ a hyperplane, and consider the $4$-dimensional cone $\rho~\in~\Tup_{0,2}(\PP^n,3)$ indexed by the following combinatorial type
\begin{center}
\begin{tikzpicture}[scale=0.8]
	\draw[fill=black] (1,-2) circle[radius=2pt];
	\draw (1,-2) node[right]{\small$v_0$};
	\draw [blue] (1,-2) node[left]{\footnotesize$0$};
	\draw[->] (1,-2) -- (1.3,-2.5);
	\draw (1.4,-2.45) node[below]{\small$x_2$};
	\draw [->] (1,-2) -- (0.7,-2.5);
	\draw (0.6,-2.45) node[below]{\small$x_1$};
	
	\draw[fill=black] (0,0) circle[radius=2pt];
	\draw (0,0) node[left]{\small$v_1$};
	\draw [blue] (0,0) node[right]{\footnotesize$1$};
	\draw[fill=black] (0,0) -- (1,-2);
	\draw (0.5,-1) node[left]{\small$e_1$};
	
	\draw[fill=black] (2,0) circle[radius=2pt];
	\draw (2,0) node[right]{\small$v_2$};
	\draw [blue] (2,0) node[left]{\footnotesize$0$};
	\draw[fill=black] (2,0) -- (1,-2);
	\draw (1.5,-1) node[right]{\small$e_2$};
	
	\draw [fill=black] (1.25,1.5) circle[radius=2pt];
	\draw (1.25,1.5) node[left]{\small$v_3$};
	\draw [blue] (1.25,1.5) node[right]{\footnotesize$1$};
	\draw (2,0) -- (1.25,1.5);
	\draw (1.7,0.8) node[left]{\small$e_3$};
	
	\draw [fill=black] (2.75,1.5) circle[radius=2pt];
	\draw (2.75,1.5) node[right]{\small$v_4$};
	\draw [blue] (2.75,1.5) node[left]{\footnotesize$1$};
	\draw (2,0) -- (2.75,1.5);
	\draw (2.4,0.8) node[right]{\small$e_4$};
\end{tikzpicture}	
\end{center}
where the degree data is given in blue. We show how the above procedure produces the radial alignment subdivision of $\rho$. We assign formal expansion factors to the edges, which in this case gives:
\begin{center}
\begin{tikzpicture}[scale=0.8]
	\draw[fill=black] (1,-2) circle[radius=2pt];
	\draw (1,-2) node[right]{\small$v_0$};
	\draw [blue] (1,-2) node[left]{\footnotesize$0$};
	\draw[->] (1,-2) -- (1.3,-2.5);
	\draw (1.4,-2.45) node[below]{\small$x_2$};
	\draw [->] (1,-2) -- (0.7,-2.5);
	\draw (0.6,-2.45) node[below]{\small$x_1$};
	
	\draw[fill=black] (0,0) circle[radius=2pt];
	\draw (0,0) node[left]{\small$v_1$};
	\draw [blue] (0,0) node[right]{\footnotesize$1$};
	\draw[fill=black] (0,0) -- (1,-2);
	\draw (0.5,-1) node[left]{\small$e_1$};
	\draw [red] (0.4,-0.9) node[right]{\footnotesize$1$};
	
	\draw[fill=black] (2,0) circle[radius=2pt];
	\draw (2,0) node[right]{\small$v_2$};
	\draw [blue] (2,0) node[left]{\footnotesize$0$};
	\draw[fill=black] (2,0) -- (1,-2);
	\draw (1.5,-1) node[right]{\small$e_2$};
	\draw [red] (1.6,-0.9) node[left]{\footnotesize$2$};
	
	\draw [fill=black] (1.25,1.5) circle[radius=2pt];
	\draw (1.25,1.5) node[left]{\small$v_3$};
	\draw [blue] (1.25,1.5) node[right]{\footnotesize$1$};
	\draw (2,0) -- (1.25,1.5);
	\draw (1.7,0.8) node[left]{\small$e_3$};
	\draw [red] (1.5,0.9) node[right]{\footnotesize$1$};
	
	\draw [fill=black] (2.75,1.5) circle[radius=2pt];
	\draw (2.75,1.5) node[right]{\small$v_4$};
	\draw [blue] (2.75,1.5) node[left]{\footnotesize$1$};
	\draw (2,0) -- (2.75,1.5);
	\draw (2.4,0.8) node[right]{\small$e_4$};
	\draw [red] (2.5,0.9) node[left]{\footnotesize$1$};
\end{tikzpicture}	
\end{center}
The distances from the root vertex are then given by:
\begin{equation*}\varphi(v_1) = e_1, \, \varphi(v_2) = 2e_2, \, \varphi(v_3) = 2e_2 + e_3, \, \varphi(v_4) = 2e_2 + e_4.\end{equation*}
The radial alignment construction subdivides $\rho$ into cones on which these quantities are totally ordered. Following the process outlined in the proof of Theorem~\ref{thm: iterative description}, the first step is to compare $\varphi(v_1)$ and $\varphi(v_2)$. This amounts to taking a weighted stellar subdivision along the floral subcone
\begin{center}
\begin{tikzpicture}[scale=0.8]
	\draw[fill=black] (1,-2) circle[radius=2pt];
	\draw (1,-2) node[right]{\small$v_0$};
	\draw [blue] (1,-2) node[left]{\footnotesize$0$};
	\draw[->] (1,-2) -- (1.3,-2.5);
	\draw (1.4,-2.45) node[below]{\small$x_2$};
	\draw [->] (1,-2) -- (0.7,-2.5);
	\draw (0.6,-2.45) node[below]{\small$x_1$};
	
	\draw[fill=black] (0,0) circle[radius=2pt];
	\draw [blue] (0,0) node[left]{\footnotesize$1$};
	\draw[fill=black] (0,0) -- (1,-2);
	\draw (0.5,-1) node[left]{\small$e_1$};
	
	\draw[fill=black] (2,0) circle[radius=2pt];
	\draw [blue] (2,0) node[left]{\footnotesize$2$};
	\draw[fill=black] (2,0) -- (1,-2);
	\draw (1.5,-1) node[right]{\small$e_2$};
\end{tikzpicture}	
\end{center}
obtained inside $\rho$ by setting $e_3=e_4=0$. The maximal cones of this first subdivision are $\rho_1=\{e_1 < 2e_2\}$ and $\rho_2=\{2e_2 < e_1\}$. We focus on the latter (similar arguments apply to the former). On $\rho_2$ we have $\varphi(v_2) < \varphi(v_1)$, and the next step is to select a minimum amongst $\varphi(v_1),\varphi(v_3)$ and $\varphi(v_4)$. This amounts to subdividing along the floral subcone obtained inside $\rho_2$ by setting $e_2=0$:
\begin{center}
\begin{tikzpicture}[scale=0.8]
	\draw[fill=black] (1,-2) circle[radius=2pt];
	\draw (1,-2) node[right]{\small$v_0$};
	\draw [blue] (1,-2) node[left]{\footnotesize$0$};
	\draw[->] (1,-2) -- (1.3,-2.5);
	\draw (1.4,-2.45) node[below]{\small$x_2$};
	\draw [->] (1,-2) -- (0.7,-2.5);
	\draw (0.6,-2.45) node[below]{\small$x_1$};
	
	\draw[fill=black] (-0.5,0) circle[radius=2pt];
	\draw [blue] (-0.5,0) node[left]{\footnotesize$1$};
	\draw[fill=black] (-0.5,0) -- (1,-2);
	\draw (0.25,-1) node[left]{\small$f_1$};
	
	\draw[fill=black] (1,0) circle[radius=2pt];
	\draw [fill=black] (1,-2) -- (1,0);
	\draw [blue] (1,0) node[left]{\footnotesize$1$};
	\draw (1.1,-1) node[left]{\small$e_3$};
	
	\draw[fill=black] (2.5,0) circle[radius=2pt];
	\draw [blue] (2.5,0) node[left]{\footnotesize$1$};
	\draw[fill=black] (2.5,0) -- (1,-2);
	\draw (1.75,-1) node[right]{\small$e_4$};
	
\end{tikzpicture}	
\end{center}
Here $f_1=e_1-2e_2$ forms part of the natural coordinate system on $\rho_2$. This second subdivision produces three maximal cones inside $\rho_2$, and restricting to any one of these we see that the third and final step is to subdivide along a floral subcone of type:
\begin{equation*}	
\begin{tikzpicture}[scale=0.8]
	\draw[fill=black] (1,-2) circle[radius=2pt];
	\draw (1,-2) node[right]{\small$v_0$};
	\draw [blue] (1,-2) node[left]{\footnotesize$1$};
	\draw[->] (1,-2) -- (1.3,-2.5);
	\draw (1.4,-2.45) node[below]{\small$x_2$};
	\draw [->] (1,-2) -- (0.7,-2.5);
	\draw (0.6,-2.45) node[below]{\small$x_1$};
	
	\draw[fill=black] (0,0) circle[radius=2pt];
	\draw [blue] (0,0) node[left]{\footnotesize$1$};
	\draw[fill=black] (0,0) -- (1,-2);
	
	\draw[fill=black] (2,0) circle[radius=2pt];
	\draw [blue] (2,0) node[left]{\footnotesize$1$};
	\draw[fill=black] (2,0) -- (1,-2);
	
\end{tikzpicture}	
\end{equation*}
Note that the ordering \eqref{order floral cones} of floral cones is respected. The height-$1$ slice is shown in Figure~\ref{fig example}.
\begin{figure}
\includegraphics[scale=0.3]{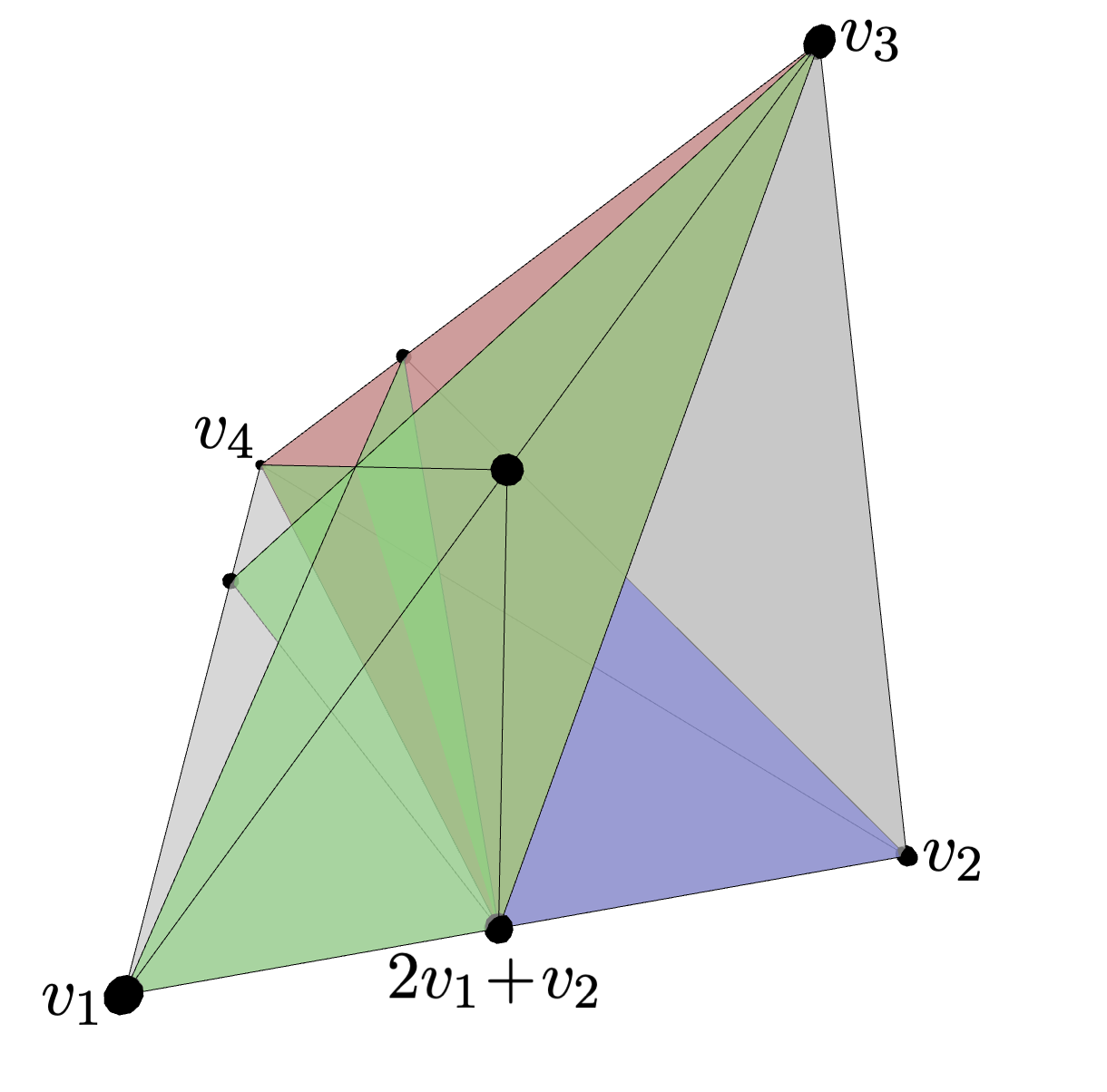}
\caption{The subdivision of $\rho$ described in Example \ref{example subdivision}. The $v_i$ are dual to the parameters $e_i$.}
\label{fig example}
\end{figure}
\end{example}

\noindent \textbf{A note on monodromy.} Floral strata typically have self-intersections. However the iterative process described above only involves blowups along strata with no self-intersections. This is the content of the following lemma; the key observation is that the self-intersection of each floral stratum is separated by taking the strict transforms along the blowups appearing earlier in the process.

\begin{lemma} Let $\sigma \in \Tup_{0,2}(X,\beta)$ be a floral cone and $\Tup_{0,2}(X,\beta)^{\ddagger}$ be the partial subdivision of $\Tup_{0,2}(X,\beta)$, such that $\sigma$ is the next floral cone to be subdivided. The stratum $Z(\sigma) \subseteq \Kup_{0,2}(X,\beta)^{\ddagger}$ has empty self-intersection.\end{lemma}

\begin{proof} Combinatorially, this says that there is no cone $\rho \in \Tup_{0,2}(X,\beta)^\ddagger$ which contains $\sigma$ as a face in two different positions. Suppose that such a cone exists. As in the proof of Theorem~\ref{thm: iterative description}, $\rho$ is indexed by a combinatorial type of tropical curve, with a partial ordering of a subset of its vertices (which without loss of generality are chosen to be strict):
\begin{equation*} \varphi(u_1) < \ldots < \varphi(u_k) < \varphi(v) \ \ \text{ for\,\, } v \not\in \{ u_1,\ldots,u_k,v_0\}.
	\end{equation*}
Let $v_1,\ldots,v_r$ denote the vertices of $\sqC$ lying immediately outside the circle of radius $\varphi(u_k)$ around $v_0$. Then a $\QQ$-coordinate system for the simplicial cone $\rho$ is given by
\begin{align} \varphi(u_1),\, \varphi(u_2) - \varphi(u_1), \, \ldots, \,\varphi(u_k)-\varphi(u_{k-1}),& \label{eqn: coord width annuli} \\
\varphi(v_1) - \varphi(u_k), \, \varphi(v_2)-\varphi(u_k), \, \ldots, \, \varphi(v_r)-\varphi(u_k), & \label{eqn: coord sigma} \\
f_1,\, \ldots, \, f_l,& \label{eqn: coord far away} \end{align}
where the functions $f_1,\ldots,f_l$ are the lengths of all edges of $\sqC$ which descend from $v_1,\ldots,v_r$.

The above parameters are presented in increasing order of distance from $v_0$. The ordering of the $\varphi(u_i)$ defines a collection of concentric circles around $v_0$, and the coordinates \eqref{eqn: coord width annuli} give the widths of the corresponding annuli. Since $\sigma$ is the next cone to subdivide, the functions \eqref{eqn: coord sigma} form a coordinate system for $\sigma$, cut out inside $\rho$ by setting the coordinates \eqref{eqn: coord width annuli} and \eqref{eqn: coord far away} to zero.

Recall we assume that $\rho$ contains $\sigma$ as a face in another position. Hence there is a different collection of coordinates on $\rho$ whose common vanishing locus also gives $\sigma$. We claim that this collection must also contain all of the coordinates \eqref{eqn: coord width annuli}; otherwise the resulting combinatorial type involves a nontrivial order relation between the edge lengths, and hence cannot give $\sigma$ since $\sigma$ was not altered by the previous subdivisions; see Remark \ref{rmk: cone not in star}.

Therefore all of the coordinates \eqref{eqn: coord width annuli} must be set to zero. For dimension reasons, at least one of the coordinates \eqref{eqn: coord sigma} must also be set to zero. As in the proof of Theorem~\ref{thm: iterative description}, it follows from stability that the resulting combinatorial type cannot be $\sigma$, since the discrete data associated to $v_0$ (curve class, number of adjacent edges, markings) must increase.\end{proof}

\section{Correcting the correspondence II: blowup formula analysis} \label{sec: blowup formula}
\noindent The previous section furnishes a sequence of birational modifications of the space of maps, resulting in a completely explicit intersection problem. We now unwind this problem, and explain how it corrects the local-logarithmic correspondence. 

\subsection{Corrected products}\label{corrected-product} Consider again the diagram of subdivided moduli spaces from \S \ref{sec: construction of subdivisions}:
\[
\begin{tikzcd}
\mathsf K_{0,2}^{\max}(X|D,\beta)^\dag \arrow{r}\arrow{d}\ar[rd,phantom,"\square"] & \mathsf K_{0,2}^{\max}(X|D_1,\beta)^\dag \arrow{d} \\
\mathsf K_{0,2}^{\max}(X|D_2,\beta)^\dag \arrow{r} & \mathsf K_{0,2}(X,\beta)^\dag.
\end{tikzcd}
\]
The entries in this diagram are all logarithmically smooth and irreducible. The fibre product is transverse over the dense locus where the logarithmic structure is trivial, yielding an equality
\begin{flalign*} && [\Kup^{\max}_{0,2}(X|D,\beta)^\dag] = [\Kup_{0,2}^{\max}(X|D_1,\beta)^\dag] \cdot [\Kup_{0,2}^{\max}(X|D_2,\beta)^\dag] && \text{in $\Kup_{0,2}(X,\beta)^\dag$} \end{flalign*}
(pushforwards have been suppressed from the notation). Pushing down along the modification $\rho \colon \Kup_{0,2}(X,\beta)^\dag \to \Kup_{0,2}(X,\beta)$, we obtain:
\begin{flalign} \label{eqn: refined product formula} && [\Kup^{\max}_{0,2}(X|D,\beta)] = \rho_\star \left( [\Kup_{0,2}^{\max}(X|D_1,\beta)^\dag] \cdot [\Kup_{0,2}^{\max}(X|D_2,\beta)^\dag] \right) && \text{in $\Kup_{0,2}(X,\beta)$}. \end{flalign}
It is natural to ask how this class relates to the naive intersection class:
\begin{flalign} \label{eqn: naive class} && [\Kup_{0,2}^{\max}(X|D_1,\beta)] \cdot [\Kup_{0,2}^{\max}(X|D_2,\beta)] && \text{in $\Kup_{0,2}(X,\beta)$}.\end{flalign}
We probe the geometry of the logarithmic modifications in order to explicitly describe the difference between these classes. The main result is Theorem \ref{thm: products comparison}, which expresses this difference as a sum of correction terms supported on excess loci.

This result compares the logarithmic and naive theories, see \cite[\S 3]{NabijouThesis}. We view this comparison as the more fundamental result; the local-logarithmic comparison arises as a consequence.

\subsection{Corrected local-logarithmic correspondence}
 Let $F \colon \Kup_{0,2}(X,\beta) \to \Kup_{0,0}(X,\beta)$ denote the forgetful morphism. The corrected local-logarithmic correspondence is obtained by pushing forward Theorem \ref{thm: products comparison} along $F$. The key is the following simple result concerning the local class:
\begin{lemma} \label{lemma: local class} The following relation holds in $\Kup_{0,0}(X,\beta)$
\begin{equation}\label{eqn: local equals naive}  F_\star \left([\Kup_{0,2}^{\max}(X|D_1,\beta)] \cdot [\Kup_{0,2}^{\max}(X|D_2,\beta)] \right) = \upalpha \cdot [\Kup_{0,0}(\OO_X(-D_1)\oplus\OO_X(-D_2),\beta)]^{\virt} \end{equation} 
where $\upalpha= (-1)^{(D_1+D_2)\cdot\beta} (D_1\cdot\beta)(D_2\cdot\beta)$. \end{lemma}
\begin{proof} This follows from a comparison of diagonals. For $i \in \{1,2\}$ consider the morphism
\begin{equation*} G_i \colon \Kup_{0,2}(X,\beta) \to \Kup_{0,1}(X,\beta) \end{equation*}
forgetting the marking $x_{\neq i}$, and let $H \colon \Kup_{0,1}(X,\beta) \to \Kup_{0,0}(X,\beta)$ be the morphism forgetting the remaining marking. Consider the tower:
\begin{center}
\begin{tikzcd}
	\Kup_{0,2}(X,\beta) \times \Kup_{0,2}(X,\beta) \ar[r,"G_1 \times G_2"] \ar[rr,bend right=12, "F\times F"] & \Kup_{0,1}(X,\beta) \times \Kup_{0,1}(X,\beta) \ar[r,"H \times H"] & \Kup_{0,0}(X,\beta) \times \Kup_{0,0}(X,\beta).
\end{tikzcd}
\end{center}
For $n \in \{0,1,2\}$ we denote the inclusion and fundamental class of the diagonal by:
\begin{equation*} \iota_{n} \colon \Kup_{0,n}(X,\beta) \hookrightarrow \Kup_{0,n}(X,\beta) \times \Kup_{0,n}(X,\beta), \qquad \Delta_{0,n} = (\iota_{n})_\star [\Kup_{0,n}(X,\beta)].\end{equation*}
These moduli spaces are unobstructed, and we have
\begin{equation} (G_1\times G_2)_\star (\Delta_{0,2}) = (H\times H)^\star (\Delta_{0,0}) \end{equation}
since this equality holds on the dense open locus where the source curve is smooth. There is also an equality:
\begin{equation*} [\Kup_{0,2}^{\max}(X|D_i,\beta)] = (G_i)^\star [\Kup_{0,1}^{\max}(X|D_i,\beta)]. \end{equation*}
From these we obtain:
\begin{align*} (\iota_0)_\star \big( F_\star \big( [\Kup_{0,2}^{\max}(& X|D_1,\beta)]  \cdot [\Kup_{0,2}^{\max}(X|D_2,\beta)] \big)\big) \\
& = (F\times F)_\star (\iota_{2})_\star \left( [\Kup_{0,2}^{\max}(X|D_1,\beta)] \cdot [\Kup_{0,2}^{\max}(X|D_2,\beta)] \right) \\
& = (F \times F)_\star \left( \left( [\Kup_{0,2}^{\max}(X|D_1,\beta)] \times [\Kup_{0,2}^{\max}(X|D_2,\beta)] \right) \cdot \Delta_{0,2} \right) \\
& = (H \times H)_\star (G_1 \times G_2)_\star \left( \left( G_1^\star[\Kup_{0,1}^{\max}(X|D_1,\beta)] \times G_2^\star [\Kup_{0,1}^{\max}(X|D_2,\beta)] \right) \cdot \Delta_{0,2} \right) \\
& = (H \times H)_\star \left( \left( [\Kup_{0,1}^{\max}(X|D_1,\beta)] \times [\Kup_{0,1}^{\max}(X|D_2,\beta)]  \right) \cdot (G_1 \times G_2)_\star (\Delta_{0,2}) \right) \\
& = (H \times H)_\star \left( \left( [\Kup_{0,1}^{\max}(X|D_1,\beta)] \times [\Kup_{0,1}^{\max}(X|D_2,\beta)] \right) \cdot (H \times H)^\star (\Delta_{0,0}) \right) \\
& = \left( H_\star [\Kup_{0,1}^{\max}(X|D_1,\beta)] \times H_\star [\Kup_{0,1}^{\max}(X|D_2,\beta)]  \right) \cdot \Delta_{0,0} \\
& = (\iota_{0})_\star \left( H_\star [\Kup_{0,1}^{\max}(X|D_1,\beta)] \cdot H_\star [\Kup_{0,1}^{\max}(X|D_2,\beta)]  \right).
\end{align*}
Letting $p \colon \Kup_{0,0}(X,\beta) \times \Kup_{0,0}(X,\beta) \to \Kup_{0,0}(X,\beta)$ be the first projection, we have $p \circ \iota_0 = \operatorname{Id}$ and so applying $p_\star$ gives:
\begin{equation*} F_\star \big( [\Kup_{0,2}^{\max}(X|D_1,\beta)]  \cdot [\Kup_{0,2}^{\max}(X|D_2,\beta)] \big) = H_\star [\Kup_{0,1}^{\max}(X|D_1,\beta)] \cdot H_\star [\Kup_{0,1}^{\max}(X|D_2,\beta)]. \end{equation*}
The claim then follows from the local-logarithmic correspondence for smooth divisors \cite{vGGR} applied to the targets $(X|D_1)$ and $(X|D_2)$, and the product formula for the local class (a consequence of the splitting of the obstruction bundle).\end{proof}
 Given this result, if we can relate the classes \eqref{eqn: refined product formula} and \eqref{eqn: naive class} on $\Kup_{0,2}(X,\beta)$, then pushing forward along $F$ will relate the logarithmic class to the local class on $\Kup_{0,0}(X,\beta)$.

\subsection{Iterated blowups: conventions and notation} \label{sec: blowup notation} The iterative description of the modification $\rho$ given in \S \ref{sec: iterative description} will allow us to access this intersection theory. The key tool is Fulton's blowup formula comparing the strict transform to the refined total transform \cite[\S 6.7]{FultonIntersectionTheory}.

By Corollary \ref{cor: iterative blowup} we may factor the modification $\rho$ as a tower of weighted blowups along strict transforms of floral strata:
\begin{equation}\label{eqn: absolute tower} \Kup_{0,2}(X,\beta)^\dag = \Kup_{0,2}(X,\beta)_m \to \Kup_{0,2}(X,\beta)_{m-1} \to \cdots \to \Kup_{0,2}(X,\beta)_0 = \Kup_{0,2}(X,\beta).\end{equation}
For $j \in \{ 1,\ldots,m\}$ we let $\sigma_j$ denote the floral cone whose weighted blowup produces $\Kup_{0,2}(X,\beta)_{j} \to \Kup_{0,2}(X,\beta)_{j-1}$. Notice that $\sigma_j$ is a cone in both $\Tup_{0,2}(X,\beta)$ and in the subdivided cone complex $\Tup_{0,2}(X,\beta)_{j-1}$, and as such represents strata in both $\Kup_{0,2}(X,\beta)$ and $\Kup_{0,2}(X,\beta)_{j-1}$. We denote these respectively by:
\begin{align*} Z({\sigma_j}) = Z(\sigma_j)_0 \subseteq \Kup_{0,2}(X,\beta), \qquad Z(\sigma_j)_{j-1} \subseteq \Kup_{0,2}(X,\beta)_{j-1}.\end{align*}
Of course, $Z(\sigma_j)_{j-1}$ is nothing but the strict transform of $Z(\sigma_j)$ under the preceding weighted blowups $\Kup_{0,2}(X,\beta)_{j-1} \to \Kup_{0,2}(X,\beta)$ (which is well-defined because of the ordering: see Remark \ref{rmk: strict transforms well defined}).

For $0 \leq j < k \leq m$ we let $\rho_{k,j}$ denote the birational morphism:
\begin{equation*} \rho_{k,j} \colon \Kup_{0,2}(X,\beta)_k \to \Kup_{0,2}(X,\beta)_j.\end{equation*}

\subsection{The blowup formula: strict and total transforms}\label{sec: excess loci}
For $i \in \{1,2\}$ there is a corresponding tower of strict transforms:
\begin{equation*} \Kup_{0,2}^{\max}(X|D_i,\beta)^\dag = \Kup_{0,2}^{\max}(X|D_i,\beta)_m \to \Kup_{0,2}^{\max}(X|D_i,\beta)_{m-1} \to \cdots \to \Kup_{0,2}^{\max}(X|D_i,\beta)_0 = \Kup_{0,2}^{\max}(X|D_i,\beta).\end{equation*}
(Note that by Corollary \ref{cor: D1 space not modified} these strict transforms are all isomorphic for $i=1$; this does not hold for $i=2$, and in any case is not important for the arguments which follow.)

Fulton's blowup formula compares the fundamental class of the strict transform to the refined fundamental class of the total transform, the latter being defined by Gysin pullback \cite[Proposition 6.7 and Example 6.7.1]{FultonIntersectionTheory}. In \S \ref{sec: weighted blowups} we explain how to extend this to weighted blowups. For each $j \in \{1,\ldots,m\}$ we obtain the following relation in $\Kup_{0,2}(X,\beta)_j$
\begin{equation}\label{eqn: blowup formula 1} \rho_{j,j-1}^\star [\Kup_{0,2}^{\max}(X|D_i,\beta)_{j-1}] = [\Kup_{0,2}^{\max}(X|D_i,\beta)_j] + A^i_j
\end{equation}
where $A^i_j$ is a correction term supported on the excess locus of the total transform:
\begin{center}
\begin{tikzcd}
\Kup_{0,2}^{\max}(X|D_i,\beta)_j \ar[r,hook] & \Kup_{0,2}^{\max}(X|D_i,\beta)_j^{\operatorname{tot}} \ar[r] \ar[d] \ar[rd,phantom,"\square"] & \Kup_{0,2}(X,\beta)_j \ar[d,"\rho_{j,j-1}"] \\
\, & \Kup_{0,2}^{\max}(X|D_i,\beta)_{j-1} \ar[r] & \Kup_{0,2}(X,\beta)_{j-1}.
\end{tikzcd}
\end{center}
We describe the terms $A^i_j$ in detail. The weighted blowup of the moduli space of ordinary stable maps has an exceptional divisor lying over the blowup centre
\begin{center}
\begin{tikzcd}
E_j \ar[r,hook] \ar[d] \ar[rd,phantom,"\square"] & \Kup_{0,2}(X,\beta)_j \ar[d] \\
Z(\sigma_j)_{j-1} \ar[r,hook] & \Kup_{0,2}(X,\beta)_{j-1}	
\end{tikzcd}	
\end{center}
and the excess locus $F^i_j$ in the total transform of the logarithmic moduli space is obtained by fibring over this blowup centre:
\begin{center}
\begin{tikzcd}
F^i_j  \ar[r] \ar[d] \ar[rd,phantom,"\square"] & E_j \ar[d] \\
\Kup_{0,2}^{\max}(X|D_i,\beta)_{j-1} \ar[r] & \Kup_{0,2}(X,\beta)_{j-1}.
\end{tikzcd}	
\end{center}
To describe $F^i_j$ we must identify the maximal strata in $\Kup_{0,2}^{\max}(X|D_i,\beta)_{j-1}$ which map to $Z(\sigma_j)_{j-1}$. This is equivalent to identifying the minimal cones in $\Tup_{0,2}^{\max}(X|D_i,\beta)_{j-1}$ that map to the interior of $\Star(\sigma_j) \subseteq \Tup_{0,2}(X,\beta)_{j-1}$.

\subsubsection{Excess locus for $D_1$} For $i=1$ these cones be identified very explicitly. We expect the following result to be extremely useful for future applications of the rank reduction technique. Recall that associated to the floral cone $\sigma_j \in \Tup_{0,2}(X,\beta)$ there is a unique cone
\begin{equation*} \tau_j \in \Tup_{0,2}^{\max}(X|D_1,\beta) \end{equation*}
obtained by assigning formal expansion factors to the edges of the tropical curve and requiring all non-root vertices to map to $0 \in \RR_{\geq 0}$:
\begin{center}
\begin{tikzpicture}

	\draw[fill=black] (0,0) circle[radius=2pt];
	\draw (0,0) node[left]{\small$\beta_1$};
	\draw[fill=black] (0,0) -- (1,-2);
		
	\draw (1,0) node{$\ldots$};
	
	\draw[fill=black] (2,0) circle[radius=2pt];
	\draw (2,0) node[right]{\small$\beta_r$};
	\draw[fill=black] (2,0) -- (1,-2);
	
	\draw[fill=black] (1,-2) circle[radius=2pt];
	\draw (1,-2) node[left]{\small$\beta_0$};
	\draw[->] (1,-2) -- (1,-2.5);
	\draw (1.05,-2.45) node[below]{\small$x_1$};
	
	\draw (1,-3.5) node[below]{$\sigma_j$};
	
	
	
	
	\draw (3.5,-1) node{\large$\rightsquigarrow$};
	
	
	\draw [fill=black] (5,0) circle[radius=2pt];
	\draw (5,0) -- (7,-0.85);
	\draw [blue] (6,-0.45) node[above]{\small$m_1$};
	
	\draw (5,-0.75) node{\vdots};
	
	\draw [fill=black] (5,-1.75) circle[radius=2pt];
	\draw (5,-1.75) -- (7,-0.85);
	\draw [blue] (6,-1.35) node[below]{\small$m_r$};
	
	\draw [fill=black] (7,-0.85) circle[radius=2pt];
	\draw [->] (7,-0.85) -- (8,-0.85);
	\draw (8,-0.8) node[above]{\small$x_1$};
	
	\draw[fill=blue] (5,-2.5) circle[radius=2pt];
	\draw[blue] (5,-2.5) node[below]{\SMALL$0$};
	\draw[->,blue] (5,-2.5) -- (8,-2.5);	
	\draw[blue] (7,-2.5) node{x};
	\draw [blue] (8,-2.5) node[above]{\small$D_1$};
	
	\draw (6.5,-3.5) node[below]{$\tau_j$};
	
\end{tikzpicture}	
\end{center}	
Here each expansion factor $m_i = D_1 \cdot \beta_i$ is non-zero (see Remark \ref{rmk: not all floral cones}) and so $\dim\tau_j = 1$. We refer to $\tau_j$ as a \textbf{comb cone}. Clearly we have $\tau_j \to \sigma_j$ and both $\tau_j$ and $\sigma_j$ are unaffected by the first $j-1$ weighted stellar subdivisions.

\begin{theorem} \label{thm: minimal D1 cone} The cone $\tau_j$ in $\Tup^{\max}_{0,2}(X|D_1,\beta)_{j-1}$ is the unique minimal cone mapping into $\Star(\sigma_j) \subseteq \Tup_{0,2}(X,\beta)_{j-1}$ i.e. every cone $\theta$ in $\Tup^{\max}_{0,2}(X|D_1,\beta)_{j-1}$ mapping into $\Star(\sigma_j)$ contains $\tau_j$ as a face.
\end{theorem}

\begin{remark} The analogous statement on the initial moduli space $\Tup_{0,2}(X,\beta)$ fails to hold; a simple counterexample with $(X|D_1)=(\PP^n|H)$ is given by

\begin{tabu}[h!]{p{0.3\textwidth} p{0.3\textwidth} p{0.3\textwidth}}

	\centering
	\begin{tikzpicture}
	\draw[fill=black] (0,0) circle[radius=2pt];
	\draw[blue] (0,0) node[left]{\small$1$};
	\draw[fill=black] (0,0) -- (1,-2);
	\draw (0.5,-1) node[left]{\small$e_1$};
	
	\draw[fill=black] (2,0) circle[radius=2pt];
	\draw[blue] (2,0) node[right]{\small$1$};
	\draw[fill=black] (2,0) -- (1,-2);
	\draw (1.5,-1) node[right]{\small$e_2$};
	
	\draw[fill=black] (1,-2) circle[radius=2pt];
	\draw[blue] (1,-2) node[left]{\small$2$};
	\draw[->] (1,-2) -- (1,-2.5);
	\draw (1.05,-2.45) node[below]{\small$x_1$};
	
	\draw (1,-3) node[below]{$\sigma_j$};
	\end{tikzpicture}

&

	\centering
	\begin{tikzpicture}
		\draw [fill=black] (5,-0.25) circle[radius=2pt];
		\draw [blue] (5,-0.25) node[left]{\small$1$};
		\draw (5,-0.25) -- (7,-0.85);
		\draw (6,-0.55) node[above]{\small$e_1$};
	
		\draw [fill=black] (5,-1.5) circle[radius=2pt];
		\draw [blue] (5,-1.5) node[left]{\small$1$};
		\draw (5,-1.5) -- (7,-0.85);
		\draw (6,-1.2) node[below]{\small$e_2$};
	
	\draw [fill=black] (7,-0.85) circle[radius=2pt];
	\draw [blue] (7,-0.85) node[above]{\small$2$};
	\draw [->] (7,-0.85) -- (8,-0.85);
	\draw (8,-0.8) node[above]{\small$x_1$};
	
	\draw[fill=blue] (5,-2.5) circle[radius=2pt];
	\draw[blue] (5,-2.5) node[below]{\SMALL$0$};
	\draw[->,blue] (5,-2.5) -- (8,-2.5);
	\draw[blue] (7,-2.5) node{x};
	\draw [blue] (8,-2.5) node[above]{\small$D_1$};
	
	\draw (6.5,-3) node[below]{$\tau_j$};
	\end{tikzpicture}
	
 &
 
	\centering
	\begin{tikzpicture}
	\draw [fill=black] (5,0) circle[radius=2pt];
	\draw [blue] (5,0) node[left]{\small$1$};
	\draw (5,0) -- (7,-0.85);
	\draw (6,-0.45) node[above]{\small$e_1$};
	
	\draw [fill=black] (5,-0.85) circle[radius=2pt];
	\draw [blue] (5,-0.85) node[left]{\small$1$};
	\draw (5,-0.85) -- (7,-0.85);
	\draw (6,-0.925) node[above]{\small$e_2$};
	
	\draw [fill=black] (5,-1.75) circle[radius=2pt];
	\draw [blue] (5,-1.75) node[left]{\small$2$};
	\draw (5,-1.75) -- (7,-0.85);
	\draw (6,-1.3) node[below]{\small$f$};
	
	\draw [fill=black] (7,-0.85) circle[radius=2pt];
	\draw [blue] (7,-0.85) node[above]{\small$0$};
	\draw [->] (7,-0.85) -- (8,-0.85);
	\draw (8,-0.8) node[above]{\small$x_1$};
	
	\draw[fill=blue] (5,-2.5) circle[radius=2pt];
	\draw[blue] (5,-2.5) node[below]{\SMALL$0$};
	\draw[->,blue] (5,-2.5) -- (8,-2.5);
	\draw[blue] (7,-2.5) node{x};
	\draw [blue] (8,-2.5) node[above]{\small$D_1$};
	
	\draw (6.5,-3) node[below]{$\theta$};
\end{tikzpicture}
\end{tabu}\\
\noindent where the curve classes are indicated in blue and the tropical edge lengths in black. Here $\theta$ maps into $\Star(\sigma_j)$ but does not contain $\tau_j$ as a face. The issue is that when we try to specialise $\theta$ by setting $f=0$, the edge length relations $e_1=e_2=2f$ force $e_1=e_2=0$ as well.

Nevertheless, we claim that after performing the first $j-1$ subdivisions, the statement holds. The reason for this is that these subdivisions cause the star of $\sigma_j$ to become smaller, so eventually every such minimal cone $\theta$ maps outside the star; this occurs because the tropical edge length relations for continuity of the tropical map to $\RR_{\geq 0}$ become incompatible with the radial alignment inequalities. Geometrically, this means that every such maximal stratum $Z(\theta)$ becomes separated from $Z(\sigma_j)$ by the process of blowing up and taking strict transforms.

This process is visible in the above example; the first subdivision reduces the star of $\sigma_j$ to the locus where $2f < e_1, e_2$, and this is incompatible with the continuity relations on $\theta$.
\end{remark}

\begin{proof}
Suppose that we are given a cone $\rho$ in $\Star(\sigma_j$):
	\begin{equation*} \sigma_j \subseteq \rho \in \Tup_{0,2}(X,\beta)_{j-1}.\end{equation*}
The cone $\rho$ is indexed by the combinatorial type of a more degenerate tropical curve, together with a partial radial alignment of its vertices. Crucially, this partial alignment is such that, in the iterative subdivision process described in the proof of Theorem \ref{thm: iterative description}, $\sigma_j$ is the next floral cone at which we must subdivide. Now consider a cone
\begin{equation*} \theta \in \Tup^{\max}_{0,2}(X|D_1,\beta)_{j-1} = \Tup^{\max}_{0,2}(X|D_1,\beta) \end{equation*}
which maps into $\rho.$ We wish to prove that $\theta$ contains the comb cone $\tau_j$ as a face.

In order for $\theta \to \rho$ their combinatorial types must have the same source curve, after possibly removing some $2$-valent vertices. The vertex alignments/orderings induced by the combinatorial types must be compatible. This means that, if the combinatorial type of $\theta$ has $k$ target expansion levels, then $\rho$ aligns precisely those vertices mapped to the final $l$ levels (for some $l < k$),  in the same order as prescribed by $\theta$.

As noted, the partial ordering on $\rho$ is such that in the iterative subdivision process, $\sigma_j$ is the next cone at which we must subdivide. Following the procedure in the proof of Theorem \ref{thm: iterative description}, this means the following: after we set $\varphi(u_i)=0$ for every vertex $u_i$ appearing in the partial alignment, we obtain a tropical curve such that the immediate descendants of $v_0$ give the type of $\sigma_j$, in the sense that setting all further edge lengths to zero specialises to the cone $\sigma_j$.

On $\theta$, the tropical parameters are given by the target edge lengths, and the above description of the shape of the tropical curve implies that when we set every target edge length except the $(k-l-1)$st to zero, we specialise to the type of $\tau_j$. This shows $\tau_j$ is a face of $\theta$, as claimed.	
\end{proof}
By Theorem \ref{thm: minimal D1 cone}, the excess locus $F^1_j$ in the total transform of $\Kup_{0,2}^{\max}(X|D_1,\beta)_{j-1}$ is given by:
\begin{equation*} F^1_j = Z(\tau_j)_{j-1} \times_{Z(\sigma_j)_{j-1}} E_j.\end{equation*}
This is an explicit weighted projective bundle, of dimension $r-1$ over the comb stratum $Z(\tau_j)_{j-1} = Z(\tau_j)$ (where $r$ is the number of leaves of the source curve in $\sigma_j$). It has excess dimension $r-2$, and the blowup formula takes the form
\begin{equation}\label{eqn: blowup formula D1} \rho_{j,j-1}^\star [\Kup_{0,2}^{\max}(X|D_1,\beta)_{j-1}] = [\Kup_{0,2}^{\max}(X|D_1,\beta)_{j}] + \gamma^1_j \cap [F_j^1]
\end{equation}
where $\gamma^1_j$ is an excess class of codimension $r-2$. This is obtained from Chern and Segre classes of strata, and hence can be described algorithmically in terms of tautological classes, see Remark~\ref{rem: effectivity}. We defer this computation to future work. 

We may now pass up the tower \eqref{eqn: absolute tower} of weighted blowups, applying the formula \eqref{eqn: blowup formula D1} at each level. At the top we obtain
\begin{flalign} \label{eqn: iterated blowup formula D1} && \rho^\star [\Kup_{0,2}^{\max}(X|D_1,\beta)] = [\Kup^{\max}_{0,2}(X|D_1,\beta)^\dag] + \sum_{j=1}^m \rho_{m,j}^\star \left( \gamma^1_j \cap [F^1_j] \right) && \end{flalign}
in $\Kup_{0,2}(X,\beta)^\dag = \Kup_{0,2}(X,\beta)_m$ (where $\rho=\rho_{m,0}$). We note that the pullback $\rho_{m,j}^\star ( \gamma^1_j \cap [F^1_j])$ is typically non-transverse, and so further excess classes will enter into the formula; however, since we will push back down along $\rho$, there is no need to describe these. This formula provides a quantitative relationship between the strict transform and the refined total transform of $\Kup_{0,2}^{\max}(X|D_1,\beta)$.

\subsubsection{Excess locus for $D_2$} For $i=2$ there is no result analogous to Theorem~\ref{thm: minimal D1 cone}, and there are typically several minimal cones
\begin{equation*} \theta \in \Tup^{\max}_{0,2}(X|D_2,\beta)_{j-1} \end{equation*}
which map into $\Star(\sigma_j)$. It is possible to construct examples where these minimal cones have different dimensions. This is not surprising; the subdivision of $\Tup_{0,2}(X,\beta)$ is constructed with reference to $D_1$, and so is typically insensitive to the geometry of $D_2$.

Let us denote the minimal cones of $\Tup_{0,2}^{\max}(X|D_2,\beta)_{j-1}$ mapping into $\Star(\sigma_j)$ by:
\begin{equation*} \theta_j(1), \ldots, \theta_j(l_j). \end{equation*}
We note that some of these cones may be exceptional, i.e. the corresponding strata
\begin{equation*} Z(\theta_j(k))_{j-1} \subseteq \Kup_{0,2}^{\max}(X|D_2,\beta)_{j-1} \end{equation*}
may have positive dimension over $\Kup_{0,2}^{\max}(X|D_2,\beta)$. The irreducible components of the excess locus $F^2_j$ are the fibre products:
\begin{equation*} F^2_j(k) = Z(\theta_j(k))_{j-1} \times_{Z(\sigma_j)_{j-1}} E_j. \end{equation*}
As before, each such component is an explicit weighted projective bundle, of dimension $r-1$ over the stratum $Z(\theta_j(k))_{j-1}$. Its excess dimension is determined by the dimension of the cone $\theta_j(k)$, and is at most $r-2$. The excess class $\gamma_j^2$ may be written as a sum of classes pushed forward from the irreducible components (such an expression is necessarily non-unique, but see Remark~\ref{rem: effectivity}). We arrive at the correction term:
\begin{equation}\label{eqn: blowup formula D2} \rho_{j,j-1}^\star [\Kup_{0,2}^{\max}(X|D_2,\beta)_{j-1}] = [\Kup_{0,2}^{\max}(X|D_2,\beta)_{j}] + \Sigma_{k=1}^{l_j} \gamma^2_j(k) \cap [F_j^2(k)]. \end{equation}
Applying this iteratively, at each level of the tower \eqref{eqn: absolute tower}, we obtain
\begin{flalign} \label{eqn: iterated blowup formula D2} && \rho^\star [\Kup_{0,2}^{\max}(X|D_2,\beta)] = [\Kup^{\max}_{0,2}(X|D_2,\beta)^\dag] + \sum_{j=1}^m \sum_{k=1}^{l_j} \rho_{m,j}^\star \left( \gamma^2_j(k) \cap [F^2_j(k)] \right) && \end{flalign}
in $\Kup_{0,2}(X,\beta)^\dag = \Kup_{0,2}(X,\beta)_m$.

\subsection{Corrected product formula and local-logarithmic correspondence} \label{sec: corrected correspondence} Expressions \eqref{eqn: iterated blowup formula D1} and \eqref{eqn: iterated blowup formula D2} allow us to compare the intersections before and after blowing up. We obtain:
\begin{align*} \rho^\star \big( [\Kup_{0,2}^{\max}(X|D_1,\beta)] \cdot [\Kup_{0,2}^{\max}(X|D_2,\beta)] \big) = \, & [\Kup_{0,2}^{\max}(X|D_1,\beta)^\dag]\cdot[\Kup^{\max}_{0,2}(X|D_2,\beta)^\dag]\,+ \\
& \rho^\star [\Kup_{0,2}^{\max}(X|D_1,\beta)] \cdot  \sum_{j=1}^m \sum_{k=1}^{l_j} \rho_{m,j}^\star \left( \gamma^2_j(k) \cap [F^2_j(k)] \right) + \\
& \rho^\star [\Kup_{0,2}^{\max}(X|D_2,\beta)] \cdot \sum_{j=1}^m \rho_{m,j}^\star \left( \gamma^1_j \cap [F^1_j] \right) - \\
& \sum_{j_1=1}^m \sum_{j_2=1}^m \left( \rho_{m,j_1}^\star \left(\gamma^1_{j_1} \cap [F^1_{j_1}] \right) \cdot \sum_{k=1}^{l_{j_2}} \rho_{m,j_2}^\star \left( \gamma^2_{j_2}(k) \cap [F^2_{j_2}(k)] \right) \right). \end{align*}
We now apply $\rho_\star$ to this. The first term on the right-hand side pushes forward to
\begin{equation*} [\Kup_{0,2}^{\max}(X|D,\beta)] \end{equation*}
by \eqref{eqn: refined product formula}. On the other hand, the second and third terms pushforward to zero by the projection formula, since the individual correction terms vanish under pushforward:
\begin{equation*} (\rho_{j,j-1})_\star \left(\gamma^1_{j} \cap [F^1_{j}] \right) = 0 , \qquad  (\rho_{j,j-1})_\star \left( \gamma^2_{j}(k) \cap [F^2_{j}(k)] \right)=0. \end{equation*}
Similarly, in the final term only products of correction terms with $j_1=j_2$ can survive. We obtain the comparison of logarithmic and naive virtual classes:
\begin{theorem} \label{thm: products comparison}The following relation holds in $\Kup_{0,2}(X,\beta)$:
\begin{align*} [\Kup_{0,2}^{\max}(X|D_1,\beta)] \cdot[\Kup_{0,2}^{\max}(X|D_2,\beta)] = [\Kup_{0,2}^{\max}(X|D,\beta)] - \sum_{j=1}^m (\rho_{j,0})_\star \left( \sum_{k=1}^{l_{j}} (\gamma^1_{j}\cap [F^1_{j}])\cdot (\gamma^2_j(k) \cap [F^2_{j}(k)]) \right). \end{align*} \end{theorem}
\noindent Pushforward along the morphism $F$ produces the corrected local-logarithmic correspondence, by Lemma \ref{lemma: local class}.

\begin{remark}[Implementation]\label{rem: effectivity}
The correction terms are well-understood as tautological classes. Excess loci are always weighted projective bundles over logarithmic strata, obtained by imposing edge length equalities in the tropical moduli. Excess \textit{classes} are handled as follows. (1) Chern classes of normal bundles to strata of the space of absolute maps are easily calculated. (2) The classes of the exceptional divisors are handled directly. (3) Segre classes of strata of (strict transforms of) the relative spaces; these are the most complex and arise as follows. The blowup centre is cut out by a monoidal ideal; the intersection with the relative space is therefore cut out by the pullback of this ideal. The Segre class is computed by Aluffi's formula for Segre classes~\cite[Theorem~1.1]{Alu16}. The output is an explicit weighted linear combination of smooth strata with tautological excess classes. The correction is therefore determined.
\end{remark}

In ongoing work, we calculate the terms explicitly and identify situations in which they vanish, establishing new cases of the numerical local-logarithmic correspondence.

\begin{remark} \label{rmk: additional markings} Theorem \ref{thm: products comparison} holds with the same proof with additional markings without tangency.\end{remark}

\subsection{Stacky subdivisions, roots and weighted blowups} \label{sec: weighted blowups} The key formula \eqref{eqn: blowup formula 1} above (along with its more specific counterparts \eqref{eqn: blowup formula D1} and \eqref{eqn: blowup formula D2}) requires a generalisation of Fulton's blowup formula to weighted blowups. For this, it is more convenient to interpret each weighted blowup as a smooth orbitoroidal embedding, rather than a logarithmically smooth toroidal embedding. We now explain this process. The basic idea is that a weighted blowup factors uniquely as a root stack followed by an ordinary blowup. We assume familiarity with the basics of toric orbifolds \cite{BorisovChenSmith, FantechiMannNironi}.

For $j \in \{1,\ldots,m\}$ the weighted stellar subdivision $\Tup_{0,2}(X,\beta)_j \to \Tup_{0,2}(X,\beta)_{j-1}$ is induced by a primitive weight vector $w_j \in \sigma_j \cap N_{\sigma_j}$ where $N_{\sigma_j}$ is the lattice of integral points. The lattice has a natural basis $v_1,\ldots,v_{r}$ dual to the edge lengths of the associated tropical curve, and we can write $w_j = m_1 v_1 + \ldots + m_{r} v_{r}$ with each $m_i > 0$. Let
\begin{equation} \label{eqn: inclusion of lattices} N_{\sigma_j}^\diamond \subseteq N_{\sigma_j} \end{equation}
be the finite-index sublattice generated by $m_1 v_1,\ldots,m_{r} v_{r}$. The triple $(\sigma_j,N_{\sigma_j},N_{\sigma_j}^\diamond)$ constitutes a stacky cone, hence corresponds to an affine toric orbifold, with isotropy given by the cokernel of the lattice inclusion \eqref{eqn: inclusion of lattices}. This globalises uniquely to a family of compatible sublattices, producing a stacky modification:
\begin{equation*} \Tup_{0,2}(X,\beta)_{j-1}^\diamond 	\to \Tup_{0,2}(X,\beta)_{j-1}.\end{equation*}
This is a stacky cone complex, i.e. a complex of stacky cones. By \cite[\S 4]{BorneVistoli} the stacky modification induces a non-representable toroidal modification:
\begin{equation*} \Kup_{0,2}(X,\beta)_{j-1}^\diamond \to \Kup_{0,2}(X,\beta)_{j-1}.\end{equation*}
This is an iterated root stack, with rooting index $m_i$ along the divisor $Z_i$ corresponding to the ray $v_i$. Let
\[ Z_i/m_i \subseteq \Kup_{0,2}(X,\beta)_{j-1}^\diamond\] 
denote the gerby divisor in the root stack; note that $m_i \cdot (Z_i/m_i)=Z_i$ which justifies the notation. The resulting space $\Kup_{0,2}(X,\beta)_{j-1}^\diamond$ is an \textbf{orbitoroidal embedding}, i.e. a pair which is locally isomorphic to a toric orbifold. 

The weight vector $w_j$ has co-ordinates $(1,\ldots,1)$ in the lattice $N_{\sigma_j}^\diamond$. We denote the stellar subdivision at $w_j$ by:
\begin{equation*} \Tup_{0,2}(X,\beta)_j \to \Tup_{0,2}(X,\beta)_{j-1}^\diamond.\end{equation*}
This produces an associated toroidal modification
\begin{equation*} \Kup_{0,2}(X,\beta)_j \to \Kup_{0,2}(X,\beta)_{j-1}^\diamond \end{equation*}
which is the blowup of the intersection of the gerby divisors $Z_i/m_i$. The composite
\begin{equation} \label{eqn: root stack followed by blowup} \Kup_{0,2}(X,\beta)_j \to \Kup_{0,2}(X,\beta)_{j-1}^\diamond \to \Kup_{0,2}(X,\beta)_{j-1} \end{equation}
is a stacky toroidal modification in the sense of \cite[\S 3.1]{Mol}, with relative coarse moduli space given by the ordinary weighted blowup. Locally, $\Kup_{0,2}(X,\beta)_j$ is the toric stack canonically associated to the simplicial toric variety obtained via the ordinary weighted blowup \cite[Theorem 4.11]{FantechiMannNironi}.

Strictly speaking, the above construction differs from the output of the weak semistable reduction algorithm, the latter being the relative coarse moduli space of the former. Crucially, however, the above construction still results in an integral and saturated morphism $\Kup_{0,2}(X|D_1,\beta)^\dag \to \Kup_{0,2}(X,\beta)^\dag$ which is all we require.

The composition \eqref{eqn: root stack followed by blowup} shows that this stacky toroidal modification is the composition of a root stack and an ordinary blowup. We obtain the weighted blowup formula by pulling back to the root stack and applying the ordinary blowup formula. Pullbacks of characteristic classes to the root stack are well-understood. Since $\Kup_{0,2}(X,\beta)$ is a smooth Deligne--Mumford stack, it follows inductively that each $\Kup_{0,2}(X,\beta)_j$ is also a smooth Deligne--Mumford stack (see Corollary~\ref{cor: iterated blowup smooth}).

\subsection{Example} \label{sec: example} We now apply the iterated blowup procedure to the plane conic example of \S \ref{sec: counter-2}, calculating the defect between local/naive and logarithmic theories. To indicate how the discussion generalises, we employ the same notation as used in \S\S \ref{sec: blowup notation}--\ref{sec: corrected correspondence}.

Consider degree $2$ logarithmic stable maps to $(\PP^2|H_1+H_2)$ with maximal tangency at two distinct markings. By \S \ref{sec: iterative description} the relevant floral cones in $\Tup_{0,2}(\PP^2,2)$ are:
\begin{equation*}
\begin{tabu}[h!]{p{0.3\textwidth} p{0.3\textwidth}}
	\centering
	\begin{tikzpicture}
	\draw[fill=black] (0,0) circle[radius=2pt];
	\draw[blue] (0,0) node[left]{\small$1$};
	\draw[fill=black] (0,0) -- (1,-2);
	\draw[->] (0,0) -- (0,-0.5);
	\draw (0.05,-0.45) node[below]{\small$x_2$};
	
	\draw[fill=black] (2,0) circle[radius=2pt];
	\draw[blue] (2,0) node[right]{\small$1$};
	\draw[fill=black] (2,0) -- (1,-2);
	
	\draw[fill=black] (1,-2) circle[radius=2pt];
	\draw[blue] (1,-2) node[left]{\small$0$};
	\draw[->] (1,-2) -- (1,-2.5);
	\draw (1.05,-2.45) node[below]{\small$x_1$};
	
	\draw (1,-3.1) node[below]{$\sigma_1$};
	\end{tikzpicture}
&

	\centering
	\begin{tikzpicture}
	\draw[fill=black] (0,0) circle[radius=2pt];
	\draw[blue] (0,0) node[left]{\small$1$};
	\draw[fill=black] (0,0) -- (1,-2);
	
	\draw[fill=black] (2,0) circle[radius=2pt];
	\draw[blue] (2,0) node[right]{\small$1$};
	\draw[fill=black] (2,0) -- (1,-2);
	
	\draw[fill=black] (1,-2) circle[radius=2pt];
	\draw[blue] (1,-2) node[left]{\small$0$};
	\draw[->] (1,-2) -- (0.7,-2.5);
	\draw (0.75,-2.45) node[below]{\small$x_1$};
	\draw[->] (1,-2) -- (1.3,-2.5);
	\draw (1.35,-2.45) node[below]{\small$x_2$};
	
	\draw (1,-3.1) node[below]{$\sigma_2$};
	\end{tikzpicture}
\end{tabu}
\end{equation*}
The iterated blowup is obtained by first blowing up $Z(\sigma_1)$, and then blowing up the strict transform of $Z(\sigma_2)$. The blowup weights are trivial because all edges have expansion factor $1$. In the notation of \S\ref{sec: blowup notation} we have:
\begin{align*} \Kup_{0,2}(\PP^2,2)_0 & = \Kup_{0,2}(\PP^2,2) \\
	\Kup_{0,2}(\PP^2,2)_1 & = \operatorname{Bl}_{Z(\sigma_1)} \Kup_{0,2}(\PP^2,2) \\
	\Kup_{0,2}(\PP^2,2)^\dag = \Kup_{0,2}(\PP^2,2)_2 & = \operatorname{Bl}_{Z(\sigma_2)_1} \operatorname{Bl}_{Z(\sigma_1)} \Kup_{0,2}(\PP^2,2)
\end{align*}
where $Z(\sigma_2)_1 \subseteq \Kup_{0,2}(\PP^2,2)_1$ is the strict transform of $Z(\sigma_2) \subseteq \Kup_{0,2}(\PP^2,2)$. Following \S \ref{sec: excess loci}, we calculate the excess loci and correction terms at each step. For $Z(\sigma_1)$, the minimal cones mapping to $\Star(\sigma_1) \subseteq \Tup_{0,2}(\PP^2,2)$ are:
\begin{equation*}
\begin{tabu}[h!]{p{0.3\textwidth} p{0.3\textwidth} p{0.3\textwidth}}
	\centering
	\begin{tikzpicture}
		\draw [fill=black] (5,-0.25) circle[radius=2pt];
		\draw [blue] (5,-0.25) node[left]{\small$1$};
		\draw (5,-0.25) -- (7,-0.85);
	
		\draw [fill=black] (5,-1.5) circle[radius=2pt];
		\draw [blue] (5,-1.5) node[left]{\small$1$};
		\draw (5,-1.5) -- (7,-0.85);
		\draw[->] (5,-1.5) -- (5,-2);
		\draw (5,-2) node[right]{\small$x_2$};
	
	\draw [fill=black] (7,-0.85) circle[radius=2pt];
	\draw [blue] (7,-0.85) node[above]{\small$0$};
	\draw [->] (7,-0.85) -- (8,-0.85);
	\draw (8,-0.8) node[above]{\small$x_1$};
	
	\draw[fill=blue] (5,-2.5) circle[radius=2pt];
	\draw[->,blue] (5,-2.5) -- (8,-2.5);
	\draw[blue] (7,-2.5) node{x};
	\draw [blue] (8,-2.5) node[above]{\small$H_1$};
	
	\draw (6.5,-3) node[below]{$\tau_1$};
	\end{tikzpicture}
	
	& 
	
		\centering
	\begin{tikzpicture}
		\draw [fill=black] (5,-0.85) circle[radius=2pt];
		\draw [blue] (5,-0.85) node[left]{\small$1$};
		\draw (5,-0.85) -- (7,-0.85);
		
		\draw [fill=black] (6,-0.85) circle[radius=2pt];
		\draw [blue] (6,-0.85) node[above]{\small$0$};
		\draw[->] (6,-0.85) -- (6,-1.35);
		\draw (6.05,-1.35) node[below]{\small$x_1$};

	\draw [fill=black] (7,-0.85) circle[radius=2pt];
	\draw [blue] (7,-0.85) node[above]{\small$1$};
	\draw [->] (7,-0.85) -- (8,-0.85);
	\draw (8,-0.8) node[above]{\small$x_2$};
	
	\draw[fill=blue] (5,-2.5) circle[radius=2pt];
	\draw[->,blue] (5,-2.5) -- (8,-2.5);
	\draw[blue] (7,-2.5) node{x};
	\draw[blue] (6,-2.5) node{x};
	\draw [blue] (8,-2.5) node[above]{\small$H_2$};
	
	\draw (6.5,-3) node[below]{$\theta_1(1)$};
	\end{tikzpicture}
	
	&
	
	\centering
		\begin{tikzpicture}
		\draw [fill=black] (5,-0.25) circle[radius=2pt];
		\draw [blue] (5,-0.25) node[left]{\small$1$};
		\draw (5,-0.25) -- (7,-0.85);
		\draw (6,-0.55) node{x};
	
		\draw [fill=black] (5,-1.5) circle[radius=2pt];
		\draw [blue] (5,-1.5) node[left]{\small$1$};
		\draw (5,-1.5) -- (6,-1.5);
		
		\draw [fill=black] (6,-1.5) circle[radius=2pt];
		\draw (6,-1.5) -- (7,-0.85);
		\draw [blue] (6,-1.5) node[above]{\small$0$};
		\draw[->] (6,-1.5) -- (6,-2);
		\draw (6,-2) node[right]{\small$x_1$};
	
	\draw [fill=black] (7,-0.85) circle[radius=2pt];
	\draw [blue] (7,-0.85) node[above]{\small$0$};
	\draw [->] (7,-0.85) -- (8,-0.85);
	\draw (8,-0.8) node[above]{\small$x_2$};
	
	\draw[fill=blue] (5,-2.5) circle[radius=2pt];
	\draw[->,blue] (5,-2.5) -- (8,-2.5);
	\draw[blue] (7,-2.5) node{x};
	\draw[blue] (6,-2.5) node{x};
	\draw [blue] (8,-2.5) node[above]{\small$H_2$};
	
	\draw (6.5,-3) node[below]{$\theta_1(2)$};
	\end{tikzpicture}
\end{tabu}
\end{equation*}
As guaranteed by Theorem~\ref{thm: minimal D1 cone}, there is a unique minimal cone $\tau_1 \leq \Tup_{0,2}^{\max}(\PP^2|H_1,2)$ mapping to $\Star(\sigma_1)$. On the other hand, we see that in this case there are two minimal cones $\theta_1(1),\theta_1(2) \leq \Tup_{0,2}^{\max}(\PP^2|H_2,2)$ mapping to $\Star(\sigma_1)$. Note in particular that $\theta_1(2)$ maps to $\Star(\sigma_1)$, but not to $\sigma_1$ itself.

The exceptional divisor $E_1 \subseteq \Kup_{0,2}(\PP^2,2)_1 = \operatorname{Bl}_{Z(\sigma_1)} \Kup_{0,2}(\PP^2,2)$ is a $\PP^1$ bundle over the stratum $Z(\sigma_1)$. Correspondingly, the excess loci
\begin{equation*} F_1^1 \subseteq \Kup_{0,2}^{\max}(\PP^2|H_1,2)_1^{\operatorname{tot}}, \quad F_1^2(1), F_1^2(2) \subseteq \Kup_{0,2}^{\max}(\PP^2|H_2,2)_1^{\operatorname{tot}}\end{equation*}
are $\PP^1$ bundles over the maximal strata mapping to $Z(\sigma_1)$:
\begin{equation*} Z(\tau_1) \subseteq \Kup_{0,2}^{\max}(\PP^2|H_1,2), \quad Z(\theta_1(1)), Z(\theta_1(2)) \subseteq \Kup_{0,2}^{\max}(\PP^2|H_2,2).\end{equation*}
Since $Z(\theta_1(1))$ and $Z(\theta_1(2))$ are strata of codimension $2$, $F_1^2(1)$ and $F_1^2(2)$ have excess dimension $-2+1=-1$, i.e. they do not carry an excess class. We conclude that the correction term arising from the first blowup vanishes.

We now blowup $Z(\sigma_2)_1$, the strict transform of $Z(\sigma_2)$. The minimal cones mapping to $\Star(\sigma_2) \subseteq \Tup_{0,2}(\PP^2,2)_1$ are:
\begin{equation*}
\begin{tabu}[h!]{p{0.3\textwidth} p{0.3\textwidth} p{0.3\textwidth}}
	\centering
	\begin{tikzpicture}
		\draw [fill=black] (5,-0.25) circle[radius=2pt];
		\draw [blue] (5,-0.25) node[left]{\small$1$};
		\draw (5,-0.25) -- (7,-0.85);
	
		\draw [fill=black] (5,-1.5) circle[radius=2pt];
		\draw [blue] (5,-1.5) node[left]{\small$1$};
		\draw (5,-1.5) -- (7,-0.85);
	
	\draw [fill=black] (7,-0.85) circle[radius=2pt];
	\draw [blue] (7,-0.85) node[above]{\small$0$};
	\draw [->] (7,-0.85) -- (8,-0.85);
	\draw (8,-0.8) node[above]{\small$x_1$};
	\draw [->] (7,-0.85) -- (7,-1.35);
	\draw (7,-1.35) node[right]{\small$x_2$};
	
	\draw[fill=blue] (5,-2.5) circle[radius=2pt];
	\draw[->,blue] (5,-2.5) -- (8,-2.5);
	\draw[blue] (7,-2.5) node{x};
	\draw [blue] (8,-2.5) node[above]{\small$H_1$};
	
	\draw (6.5,-3) node[below]{$\tau_2$};
	\end{tikzpicture}
	
	& 
	
	\centering
	\begin{tikzpicture}
		\draw [fill=black] (5,-0.25) circle[radius=2pt];
		\draw [blue] (5,-0.25) node[left]{\small$1$};
		\draw (5,-0.25) -- (7,-0.85);
	
		\draw [fill=black] (5,-1.5) circle[radius=2pt];
		\draw [blue] (5,-1.5) node[left]{\small$1$};
		\draw (5,-1.5) -- (7,-0.85);
	
	\draw [fill=black] (7,-0.85) circle[radius=2pt];
	\draw [blue] (7,-0.85) node[above]{\small$0$};
	\draw [->] (7,-0.85) -- (8,-0.85);
	\draw (8,-0.8) node[above]{\small$x_2$};
	\draw [->] (7,-0.85) -- (7,-1.35);
	\draw (7,-1.35) node[right]{\small$x_1$};
	
	\draw[fill=blue] (5,-2.65) circle[radius=2pt];
	\draw[->,blue] (5,-2.65) -- (8,-2.65);
	\draw[blue] (7,-2.65) node{x};
	\draw [blue] (8,-2.65) node[above]{\small$H_2$};
	
	\draw (6.5,-3) node[below]{$\theta_2(1)$};
	\end{tikzpicture}
\end{tabu}
\end{equation*}
The exceptional divisor $E_2$ in $\Kup_{0,2}(\PP^2,2)_2$ is again a $\PP^1$ bundle over $Z(\sigma_2)_1$ and so the excess loci
\begin{equation*} F_2^1 \subseteq \Kup_{0,2}^{\max}(\PP^2|H_1,2)_2^{\operatorname{tot}}, \quad F_2^2(1) \subseteq \Kup_{0,2}^{\max}(\PP^2|H_2,2)_2^{\operatorname{tot}} \end{equation*}
are $\PP^1$ bundles over $Z(\tau_2)_1$ and $Z(\theta_2(1))_1$. Both these strata have codimension $1$, so both excess loci have excess dimension $0$, and in each case the excess class is simply the fundamental class of the excess locus. We conclude that the correction term in $\Kup_{0,2}(\PP^2,2)_2$ is the product:
\begin{equation*} [F_2^1] \cdot [F_2^2(1)].\end{equation*}
Denoting the blowup morphisms by
\begin{equation*} \Kup_{0,2}(\PP^2,2)_2 \xrightarrow{\rho_{2,1}} \Kup_{0,2}(\PP^2,2)_1 \xrightarrow{\rho_{1,0}} \Kup_{0,2}(\PP^2,2)\end{equation*}
we have from Theorem~\ref{thm: products comparison}:
\begin{equation*} [\Kup_{0,2}^{\max}(\PP^2|H_1,2)] \cdot [\Kup_{0,2}^{\max}(\PP^2|H_2,2)] = [\Kup_{0,2}^{\max}(\PP^2|H_1+H_2,2)] - (\rho_{1,0})_\star(\rho_{2,1})_\star  \left( [F_2^1] \cdot [F_2^2(1)] \right).\end{equation*}
It remains to calculate the final term. Let $i \colon E_2 \hookrightarrow \Kup_{0,2}(\PP^2,2)_2$ denote the inclusion and let $\pi \colon E_2 \to Z(\sigma_2)_1$ denote the bundle projection. We have
\begin{equation*} [F_2^1] = i_\star \pi^\star [Z(\tau_2)_1], \quad [F_2^2(1)] = i_\star \pi^\star [Z(\theta_2(1))_1]\end{equation*}
from which we obtain
\begin{equation*} [F_2^1] \cdot [F_2^2(1)] = i_\star \left( -H \cap \pi^\star ( [Z(\tau_2)_1] \cdot [Z(\theta_2(1))_1]  ) \right) \end{equation*}
where $H=-c_1(N_{E_2})$ is the fibrewise hyperplane class of the projective bundle. Using $\pi_\star H = 1$ and the projection formula we obtain
\begin{equation*} (\rho_{2,1})_\star \left( [F_2^1] \cdot [F_2^2(1)] \right) = - j_\star \left( [Z(\tau_2)_1] \cdot [Z(\theta_2(1))_1] \right)\end{equation*}
where $j \colon Z(\sigma_2)_1 \hookrightarrow \Kup_{0,2}(\PP^2,2)_1$ is the inclusion. The intersection $Z(\sigma_1) \cap Z(\sigma_2)$ is a divisor in $Z(\sigma_2)$, and consequently the strict transform $Z(\sigma_2)_1 \to Z(\sigma_2)$ is an isomorphism. We thus have
\begin{equation*} (\rho_{0,1})_\star (\rho_{2,1})_\star \left( [F_2^1] \cdot [F_2^2(1)] \right) = - k_\star \left( [Z(\tau_2)] \cdot [Z(\theta_2(1))] \right) \end{equation*}
where $k \colon Z(\sigma_2) \hookrightarrow \Kup_{0,2}(\PP^2,2)$ is the inclusion. The intersection of $Z(\tau_2)$ and $Z(\theta_2(1))$ inside $Z(\sigma_2)$ is transverse. We denote this locus by:
\begin{equation*} W = Z(\tau_2) \cap Z(\theta_2(1)) \subseteq Z(\sigma_2) \subseteq \Kup_{0,2}(\PP^2,2). \end{equation*}
Geometrically, it parametrises pairs of lines through the point $H_1 \cap H_2$ joined to a contracted component of the source curve containing both markings. It has dimension $3$, and we conclude:
\begin{equation} \label{eqn: correction formula example} [\Kup_{0,2}^{\max}(\PP^2|H_1,2)] \cdot [\Kup_{0,2}^{\max}(\PP^2|H_2,2)] = [\Kup_{0,2}^{\max}(\PP^2|H_1+H_2,2)] + [W].\end{equation}
This precisely quantifies the difference between the local/naive and logarithmic theories. In this case the formula reflects the geography of the naive space, which consists of two irreducible components corresponding to the two terms on the right-hand side. On the other hand, the iterated blowup procedure gives a general-purpose algorithm which does not rely on ad hoc descriptions of the naive space.

It is easy to find insertions which pair nontrivially with $[W]$. Introduce two additional markings with no tangency conditions and consider the forgetful morphism $F \colon \Kup_{0,4}(\PP^2,2) \to \Kup_{0,2}(\PP^2,2)$. Applying $F^\star$ to \eqref{eqn: correction formula example} gives:
\begin{equation*} [\Kup_{0,4}^{\max}(\PP^2|H_1,2)] \cdot [\Kup_{0,4}^{\max}(\PP^2|H_2,2)] = [\Kup_{0,4}^{\max}(\PP^2|H_1+H_2,2)] + F^\star[W].\end{equation*}
We cap this with the insertion $\gamma = \psi_1 \ev_3^\star (\pt) \ev_4^\star(\pt)$ on $\Kup_{0,4}(\PP^2,2)$. The correction term is
\begin{equation}\label{eqn: example correction term} \gamma \cap F^\star[W] = \psi_1 \cap [\Mbar_{0,4}] = 1\end{equation}
as the point constraints fix the moduli of the two lines. On the other hand the strong form of the local-logarithmic correspondence for smooth pairs gives:
\begin{align*} \gamma \cap [\Kup_{0,4}^{\max}(\PP^2|H_1,2)] \cdot [\Kup_{0,4}^{\max}(\PP^2|H_2,2)] & =  \gamma \cdot \ev_1^\star(H)  \ev_2^\star(H) \cap [\Kup_{0,4}(\OO_{\PP^2}(-1)^{\oplus 2},2)]^{\virt} \\
& = \ev_1^\star (H) \psi_1 \ev_2^\star (H) \ev_3^\star(\pt) \ev_4^\star(\pt)\, \mathrm{e}(\R^1 \pi_\star f^\star \OO_{\PP^2}(-1))^2 \cap [\Kup_{0,4}(\PP,2)]. \end{align*}
We compute this by torus localisation. Let $H_0,H_1,H_2 \subseteq \PP^2$ be the coordinate hyperplanes and $p_0,p_1,p_2 \in \PP^2$ be the coordinate points. We choose the following equivariant lifts of the insertions:
\begin{equation*} \ev_1^\star (H_0) \psi_1 \ev_2^\star (H_1) \ev_3^\star(p_0) \ev_4^\star(p_1). \end{equation*}
We equip the first copy of $\OO_{\PP^2}(-1)$ with the torus action which has weight zero at $p_0$, and equip the second copy of $\OO_{\PP^2}(-1)$ with the torus action which has weight zero at $p_1$. Under these choices of weights and equivariant insertions, the only graph contributing to the localised integral is
 \begin{equation*}
	\begin{tikzpicture}
		\draw [fill=black] (5,-0.25) circle[radius=2pt];
		\draw (5,-0.25) node[left]{\small$p_0$};
		\draw (5,-0.25) -- (5,0.25);
		\draw (5,0.2) node[above]{\small$x_3$};
		\draw (5,-0.25) -- (7,-0.85);
		\draw (6,-0.55) node[above]{\small$1$};
	
		\draw [fill=black] (5,-1.5) circle[radius=2pt];
		\draw (5,-1.5) node[left]{\small$p_1$};
		\draw (5,-1.5) -- (5,-2);
		\draw (5,-1.95) node[below]{\small$x_4$};
		\draw (5,-1.5) -- (7,-0.85);
		\draw (6,-1.2) node[below]{\small$1$};
	
	\draw [fill=black] (7,-0.85) circle[radius=2pt];
	\draw (7,-0.85) node[right]{\small$p_2$};
	\draw (7,-0.85) -- (7,-0.35);
	\draw (7,-0.4) node[above]{\small$x_1$};
	\draw (7,-0.85) -- (7,-1.35);
	\draw (7,-1.3) node[below]{\small$x_2$};

	\end{tikzpicture}
\end{equation*}
and a direct calculation of its contribution gives:
\begin{equation}\label{eqn: example naive invariant} \gamma \cap [\Kup_{0,4}^{\max}(\PP^2|H_1,2)] \cdot [\Kup_{0,4}^{\max}(\PP^2|H_2,2)] = 1. \end{equation}
Combining \eqref{eqn: example naive invariant} and \eqref{eqn: example correction term} with \eqref{eqn: correction formula example}, we obtain the logarithmic invariant:
\begin{equation*} \gamma \cap [\Kup_{0,4}^{\max}(\PP^2|H_1+H_2,2)] = 1-1 = 0.\end{equation*}
This value can be independently verified using heuristic arguments relating the logarithmic invariant to tropical curve counts.

\section{Virtual pullback} \label{sec: virtual pullback}

\noindent We employ virtual pullback techniques to extend Theorem~\ref{thm: products comparison} to general hyperplane sections. 

\subsection{Setup} Consider a pair $(Z|E)$ with $E=E_1+E_2$, and each $E_i$ a hyperplane section. We have embeddings
\begin{equation*} (Z|E_i)\hookrightarrow (\PP^{n_i}|H_i) \end{equation*}
where $H_i$ is a hyperplane. Let $X=\PP^{n_1}\times\PP^{n_2}$ and $D=D_1+D_2$ be the simple normal crossings divisor induced by the $H_i$. There is a closed embedding $Z \hookrightarrow X$ with $E_i = Z \cap D_i$.
\begin{lemma}
The following morphisms of moduli spaces are strict:
\begin{align*} \Kup_{0,2}(Z,\beta) & \to \Kup_{0,2}(X,\beta),\\
\Kup^{\max}_{0,2}(Z|E_i,\beta) & \to \Kup^{\max}_{0,2}(X|D_i,\beta), \ \ \textrm{for $i \in \{1,2\}$}, \\
\Kup^{\max}_{0,2}(Z|E,\beta) & \to \Kup^{\max}_{0,2}(X|D,\beta).\end{align*}
They carry relative perfect obstruction theories given (in every case) by
\begin{equation*} (\pi_\star f^\star \mathrm{N}_{Z|X})^\vee[1]\end{equation*}
and the induced virtual pullback morphism identifies the virtual fundamental classes.
\end{lemma}

\begin{proof} The maps $(Z|E_i) \to (X|D_i)$ and $(Z|E) \to (X|D)$ are strict closed embeddings, so the logarithmic normal bundle coincides with the ordinary normal bundle. The obstruction theories are both relative to the moduli space of maps to the Artin fan~\cite{AW}, so the claim follows from functoriality of virtual pullbacks \cite{Mano12}. The obstruction theory is perfect due to the convexity of $X$. \end{proof}

\begin{lemma} The following square is cartesian
\begin{center}
\begin{tikzcd}
\Kup^{\max}_{0,2}(Z|E,\beta) \ar[r,"j"] \ar[d] \ar[rd,phantom,"\square"] & \Kup^{\max}_{0,2}(X|D,\beta) \ar[d] \\
\Kup_{0,2}(Z,\beta) \ar[r,"i"] & \Kup_{0,2}(X,\beta)
\end{tikzcd}	
\end{center}
and satisfies:
\begin{equation*} [\Kup^{\max}_{0,2}(Z|E,\beta)]^{\virt} = i^![\Kup^{\max}_{0,2}(X|D,\beta)].
\end{equation*}
\noindent The analogous statements hold for $(Z|E_i) \to (X|D_i)$.
\end{lemma}
\begin{proof} Since $i$ is strict, the square is a cartesian diagram of ordinary stacks. Equality of virtual classes holds as the relative perfect obstruction theory for $i$ pulls back to give the relative perfect obstruction theory for $j$.
\end{proof}

\subsection{Virtual birational models} The blowups in \S \ref{sec: blowup formula} may now be pulled back. For $j\in\{1,\ldots,m\}$ we obtain virtual birational models:
\begin{center}
\begin{tikzcd}
	\Kup_{0,2}(Z,\beta)_j \ar[r] \ar[d] \ar[rd,phantom,"\square"] & \Kup_{0,2}(X,\beta)_j \ar[d] \\
	\Kup_{0,2}(Z,\beta) \ar[r,"i"] & \Kup_{0,2}(X,\beta).
\end{tikzcd}
\end{center}
Since $i \colon \Kup_{0,2}(Z,\beta) \to \Kup_{0,2}(X,\beta)$ is strict, the morphism
\begin{equation}\label{eqn: birational model Z} \Kup_{0,2}(Z,\beta)_j \to \Kup_{0,2}(Z,\beta)\end{equation}
is also a logarithmic modification. The space $\Kup_{0,2}(Z,\beta)_j$ carries a natural perfect obstruction theory, and the pushforward morphism \eqref{eqn: birational model Z} identifies the virtual classes \cite[\S 3.5]{R19}.

We similarly obtain virtual strict transforms:
\begin{equation*}
	\Kup_{0,2}^{\max}(Z|E_i,\beta)_m \to \Kup_{0,2}^{\max}(Z|E_i,\beta)_{m-1} \to  \cdots \to \Kup_{0,2}^{\max}(Z|E_i,\beta)_0 = \Kup_{0,2}^{\max}(Z|E_i,\beta)
\end{equation*}
 for $i\in\{1,2\}$, and a comparison of obstruction theories gives:
\begin{equation*} [\Kup_{0,2}^{\max}(Z|E_i,\beta)_j]^{\virt} = i^! [\Kup_{0,2}^{\max}(X|D_i,\beta)_j].\end{equation*}

\subsection{Corrected product formula} We first restrict to the case where $Z$ is convex, so that products in $\Kup_{0,2}(Z,\beta)_j$ are well-defined. Applying the virtual pullback $i^!$ to the blowup formulae \eqref{eqn: blowup formula D1} and \eqref{eqn: blowup formula D2} results in the following relations in $\Kup_{0,2}(Z,\beta)_j$:
\begin{align} \rho_{j,j-1}^\star [\Kup_{0,2}^{\max}(Z|E_1,\beta)_{j-1}]^{\virt} & = [\Kup_{0,2}^{\max}(Z|E_1,\beta)_{j}]^{\virt} + \gamma^1_j \cap [F_j^1]^{\virt}, \\
 \rho_{j,j-1}^\star [\Kup_{0,2}^{\max}(Z|E_2,\beta)_{j-1}]^{\virt} & = [\Kup_{0,2}^{\max}(Z|E_2,\beta)_{j}]^{\virt} + \Sigma_{k=1}^{l_j} \gamma^2_j(k) \cap [F_j^2(k)]^{\virt}.	
\end{align}
As in \S\S\ref{sec: excess loci} and \ref{sec: corrected correspondence}, we now pass up the tower of logarithmic blowups, take the product and then push back down to $\Kup_{0,2}(Z,\beta)$, obtaining:
\begin{theorem} \label{thm: correct correspondence general} The following relation holds in $\Kup_{0,2}(Z,\beta)$:
\begin{align*} [\Kup_{0,2}^{\max}(Z|E_1,\beta)]^{\virt} \cdot[\Kup_{0,2}^{\max}(Z|E_2,\beta)]^{\virt} = [\Kup_{0,2}^{\max}(Z|E,\beta)]^{\virt} - \sum_{j=1}^m (\rho_{j,0})_\star \left( \sum_{k=1}^{l_{j}} (\gamma^1_{j}\cap [F^1_{j}]^{\virt})\cdot (\gamma^2_j(k) \cap [F^2_{j}(k)]^{\virt})\right). \end{align*}
Applying $F_\star$ gives the corrected local-logarithmic correspondence.
\end{theorem}
If $Z$ is not convex the Chow groups of $\Kup_{0,2}(Z,\beta)$ need not admit a product, and the corrected product formula cannot even be formulated. Instead, we apply $i^!$ to Theorem~\ref{thm: products comparison} to obtain:
\begin{align*} i^! \left( [\Kup_{0,2}^{\max}(X|D_1,\beta)] \cdot[\Kup_{0,2}^{\max}(X|D_2,\beta)] \right) = [\Kup_{0,2}^{\max}(Z|E,\beta)]^{\virt} - i^! \sum_{j=1}^m (\rho_{j,0})_\star \left( \sum_{k=1}^{l_{j}} (\gamma^1_{j}\cap [F^1_{j}])\cdot (\gamma^2_j(k) \cap [F^2_{j}(k)]) \right). \end{align*}
The strong form of the local-logarithmic correspondence for $(X|D_1)$ and $(X|D_2)$ can be used to identify the left-hand side with the local theory of $\OO_Z(E_1)\oplus\OO_Z(E_2)$ capped with $\ev_1^\star E_1 \cdot \ev_2^\star E_2$. Applying $F_\star$ we again obtain the corrected local-logarithmic correspondence. A difference with Theorem \ref{thm: correct correspondence general} is that the correction terms are calculated in $\Kup_{0,2}(X,\beta)$ and then pulled back.

\section{Local-logarithmic on product geometries} \label{sec: product-section}
\noindent We establish an instance of the numerical local-logarithmic correspondence for products. The argument here is elementary, and independent of the more technical blowup arguments elsewhere in the paper. However the understanding of the blowup geometry leads very naturally to the proof; we simply look for situations where the correction terms can be shown to vanish.

\subsection{Setup: unobstructed case} 
Let $X_1,\ldots X_k$ be smooth projective varieties equipped with smooth hyperplane sections $D_1,\ldots, D_k$. As before, we first specialise to the situation where each $X_i$ is a projective space $\PP^{n_i}$ and each $D_i=H_i$ is a hyperplane. Let
\[
\PP \colonequals \prod_{i=1}^k \PP^{n_i}
\]
be the target, $H$ the union of pullbacks of hyperplanes $H_i$ from the factors, and $\beta$ the curve class. 

We work with the space of $k+3$ pointed maps. The final three points will be taken to have zero contact order and each of the first $k$ points will have maximal contact order with the corresponding divisor. We have the following composition of forgetful morphisms:
\[
\mathsf K^{\mathrm{max}}_{0,k+3}(\PP|H,\beta)\to \mathsf K^{\mathrm{max}}_{0,k+3}(\PP^{n_i}|H_i,\beta_i)\to \mathsf K^{\mathrm{max}}_{0,4}(\PP^{n_i}|H_i,\beta_i).
\]
The first arrow projects onto the appropriate factor and stabilises the map; the second arrow forgets all marked points except $x_i$ and the three markings with zero contact order. These give rise to a morphism
\[
\rho \colon \mathsf K^{\mathrm{max}}_{0,k+3}(\PP|H,\beta)\to \prod_{i=1}^k \mathsf K^{\mathrm{max}}_{0,4}(\PP^{n_i}|H_i,\beta_i).
\]

\begin{proposition}
The morphism $\rho$ is proper and birational. 
\end{proposition}

\begin{proof}
Arguments as in \S \ref{sec: counterexamples} show that the source and target of $\rho$ are irreducible. Examine the locus in the source comprising maps from smooth domains, dimensionally transverse to $H$. Notice that the source curve is a parametrised $\PP^1$ with the parametrisation given by the three markings with zero contact order; this locus is dense, and $\rho$ has an inverse on this locus.\end{proof}


\subsection{Primary theory with factorwise insertions} \label{sec: primary factorwise insertions} Consider the morphism
\begin{equation*} \nu \colon \Kup_{0,3}(\PP,\beta) \to \prod_{i=1}^k \Kup_{0,3}(\PP^{n_i},\beta_i) \end{equation*}
of spaces of ordinary stable maps. It is clear that $\nu$ is proper and birational. We assemble primary insertions on $\Kup_{0,3}(\PP,\beta)$ without appealing to the existence of marked points. This is likely well-known to experts. Consider the universal family:
\[
\begin{tikzcd}
\mathcal C\arrow{d}{\pi}\arrow{r}{f}& \PP\\
\mathsf \Kup_{0,3}(\PP,\beta).&
\end{tikzcd}
\]
Given a cohomology class $\gamma$ in the target $\PP$ we obtain a cycle class $\pi_\star f^\star\gamma$ on $\Kup_{0,3}(\PP,\beta)$. Primary invariants are defined by integrating products of such classes. The comparison of diagonals (see the proof of Lemma~\ref{lemma: local class}) equates these integrals with the ordinary Gromov--Witten invariants:
\begin{equation*} \Pi_{j=1}^r \pi_\star f^\star \gamma_j \cap [\Kup_{0,3}(\PP,\beta)] = \Pi_{j=1}^r \ev_j^\star \gamma_j \cap [\Kup_{0,r+3}(\PP,\beta)].\end{equation*}
The three auxiliary marked points can be removed by attaching divisorial insertions and appealing to the divisor axiom; alternatively, they may be equipped with arbitrary insertions.

We now restrict to a particular class of insertions: we require that each class $\gamma_j$ is equal to the pullback of a class in $\PP^{n_i}$ along one of the projections $\PP \to \PP^{n_i}$. We refer to this as the \textbf{primary theory with factorwise insertions}. The three auxiliary markings are allowed to carry arbitrary classes; the factorwise constraint only applies to additional markings introduced via the above procedure. We assemble these insertions into a single class on $\Kup_{0,3}(\PP,\beta)$:
\begin{equation} \label{eqn: factorwise insertion class}	 \gamma = \prod_{i=1}^3 \ev_i^\star \delta_i \cdot \prod_{j=1}^r \pi_\star f^\star \gamma_j.\end{equation}

\begin{theorem}[Local-logarithmic correspondence with primary factorwise insertions] \label{thm: product conjecture projective space} If $\gamma$ is the class \eqref{eqn: factorwise insertion class} with factorwise insertions, then there is an equality
\begin{equation*} \gamma \cap \psi_\star [\Kup_{0,k+3}^{\max}(\PP|H,\beta)] = \upalpha \cdot \gamma \cap [\Kup_{0,3}(\oplus_{i=1}^k \OO_{\PP}(-H_i),\beta)]^{\virt} \end{equation*}
where $\psi \colon \Kup_{0,k+3}^{\max}(\PP|H,\beta) \to \Kup_{0,3}(\PP,\beta)$ and $\upalpha = \Pi_{i=1}^k (-1)^{d_i+1} d_i$.
\end{theorem}

\begin{proof} There is a commutative diagram
\begin{center}
\begin{tikzcd}
 \mathsf K^{\mathrm{max}}_{0,k+3}(\PP|H,\beta) \ar[r,"\rho"] \ar[d,"\psi"] & \prod_{i=1}^k \mathsf K^{\mathrm{max}}_{0,4}(\PP^{n_i}|H_i,\beta_i) \ar[d,"\varphi"] \\
	\Kup_{0,3}(\PP,\beta) \ar[r,"\nu"] & \prod_{i=1}^k \Kup_{0,3}(\PP^{n_i},\beta_i)
\end{tikzcd}	
\end{center}
with $\rho$ and $\nu$ birational. Because $\gamma$ is assembled from primary factorise insertions, it follows that $\gamma=\nu^\star \delta$ for some class $\delta$. This can be seen by comparing the universal curve over $\Kup_{0,3}(\PP,\beta)$ to the pullback of the universal curve over $\Kup_{0,3}(\PP^{n_i},\beta_i)$.

The product formula~\cite{Beh97} applied to the total spaces of $\OO_{\PP^{n_i}}(-H_i)$ shows that the local class associated to $\oplus_{i=1}^k \OO_{\PP}(-H_i)$ pushes forward along $\nu$ to the product of the local classes associated to $\OO_{\PP^{n_i}}(-H_i)$. Combining with the local-logarithmic correspondence for the smooth pairs $(\PP^{n_i},H_i)$ gives:
\begin{align*} \varphi_{\star} \left( \Pi_{i=1}^k [\Kup_{0,4}^{\max}(\PP^{n_i}|H_i,\beta_i)] \right) = \upalpha \cdot \Pi_{i=1}^k [\Kup_{0,3}(\OO_{\PP^{n_i}}(-H_i),\beta_i)]^{\virt} = \upalpha \cdot \nu_{\star} [\Kup_{0,3}(\oplus_{i=1}^k \OO_{\PP}(-H_i),\beta)]^{\virt}.\end{align*}
From this, and the projection formula applied to $\nu$, we conclude:
\begin{align*} \gamma \cap \psi_\star [\Kup_{0,k+3}^{\max}(\PP|H,\beta)] & = \delta \cap \varphi_\star \rho_\star [\Kup_{0,k+3}^{\max}(\PP|H,\beta)] \\
& = \delta \cap \varphi_\star \left( \Pi_{i=1}^k [\Kup_{0,4}^{\max}(\PP^{n_i}|H_i,\beta_i)] \right) \\
& = \upalpha \cdot \delta \cap \nu_\star [\Kup_{0,3}(\oplus_{i=1}^k  \OO_\PP(-H_i),\beta)]^{\virt} \\
& = \upalpha \cdot \gamma \cap [\Kup_{0,3}(\oplus_{i=1}^k \OO_\PP(-H_i),\beta)]^{\virt}. \qedhere \end{align*}\end{proof}

\begin{remark} The above result does not contradict Theorem~\ref{thm: products comparison}. Rather, the factorwise insertions kill the correction terms in this setting. The same phenomenon explains other cases where numerical forms of the local-logarithmic correspondence are known to hold \cite{BBvG,BBvG2}.\end{remark}

\subsection{Virtual pullback} In light of the preceding result, note that for an arbitrary section pair $(X|D)$ of product type, an identical construction produces a cartesian diagram:
\[
\begin{tikzcd}
\mathsf K^{\mathrm{max}}_{0,k+3}(X|D,\beta)\arrow{d}\arrow{r} \ar[rd,phantom,"\square"]& \mathsf K^{\mathrm{max}}_{0,k+3}(\PP|H,\beta) \arrow{d} \\
\mathsf \prod_{i=1}^k \Kup_{0,4}^{\max}(X_i|D_i,\beta_i) \arrow{r} & \mathsf \prod_{i=1}^k \Kup_{0,4}^{\max}(\PP^{n_i}|H_i,\beta_i). 
\end{tikzcd}
\]
The horizontal arrows possess compatible perfect obstruction theories, as in \S \ref{sec: virtual pullback}. The right vertical map is birational, therefore we conclude that the left vertical arrow identifies virtual classes. The proof of Theorem~\ref{thm: product conjecture projective space} then applies verbatim, extending the correspondence to section pairs: 
\begin{equation*}\gamma \cap \psi_\star [\Kup_{0,k+3}^{\max}(X|D,\beta)]^{\virt} = \upalpha \cdot \gamma \cap [\Kup_{0,3}(\oplus_{i=1}^k \OO_{X}(-D_i),\beta)]^{\virt}.\end{equation*}

\begin{remark}
We do not believe that the product structure is the true reason for the result; products only produce a birational morphism that kills the corrections. The morphism 
\begin{equation*}\Kup^{\max}_{0,k+3}(\PP|H,\beta) \to \prod_{i=1}^k \Kup^{\max}_{0,4}(\PP^{n_i},\beta_i)\end{equation*}
is a contraction. On the right-hand side there is not necessarily a universal map to $\PP$. This is analogous to the quasimap moduli. A study of naive and logarithmic quasimap theory may be worthwhile.\end{remark}

\bibliographystyle{siam} 
\bibliography{Bibliography} 

\end{document}